\documentclass[a4paper,11pt]{article}

\usepackage{amsmath, amssymb, color}
\usepackage{amsthm} 
 \usepackage{natbib}
\usepackage[colorlinks,citecolor=blue,urlcolor=blue]{hyperref}
\usepackage{xcolor,colortbl}

\RequirePackage{hypernat}
\usepackage{cleveref, enumitem}
\usepackage{color}
\usepackage{amsfonts, fancybox}
\usepackage{amssymb}
\usepackage[american]{babel}
\usepackage{psfrag,color}
\usepackage{dcolumn}
\usepackage[format=hang, justification=justified, singlelinecheck=false]{subcaption}
\usepackage{flafter, bm, dsfont}
\usepackage[section]{placeins}

\usepackage{multirow}
\usepackage{color}
\usepackage{amsmath}
\usepackage{float}
\usepackage{mathrsfs}
\usepackage{amsfonts}
\usepackage{dsfont}
\usepackage{amssymb}
\usepackage{algorithmic, algorithm}
\usepackage{wrapfig}

\usepackage{amssymb, latexsym}

\usepackage{graphicx}
\usepackage{libertine}
\usepackage[T1]{fontenc}
\usepackage[font={small}]{caption}

\newcommand{\eps}{\varepsilon}
\newcommand{\HIBM}{\mathcal{H}^{\textsc{ibm}}_{\alpha, \beta}}
\newcommand{\gCW}{g^{\textsc{cw}}}
\newcommand{\muCW}{\mu^{\textsc{CW}}}
\newcommand{\dhyp}{\{-1, 1\}^p}
\newcommand{\ubar}[1]{\underline{#1}}
\newcommand{\tr}{\mathbf{Tr}}

\numberwithin{equation}{section}

\newtheorem{definition}{Definition}[section]
\newtheorem{theorem}[definition]{Theorem}
\newtheorem{lemma}[definition]{Lemma}
\newtheorem{corollary}[definition]{Corollary}
\newtheorem{remark}[definition]{Remark}
\newtheorem{proposition}[definition]{Proposition}

\newcommand{\bi}{^{(t)}}
\newcommand{\DS}{\displaystyle}
\newcommand{\cA}{\mathcal{A}}
\newcommand{\cB}{\mathcal{B}}

\newcommand{\cE}{\mathcal{E}}

\newcommand{\cH}{\mathcal{H}}

\newcommand{\cM}{\mathcal{M}}
\newcommand{\cN}{\mathcal{N}}

\newcommand{\cP}{\mathcal{P}}

\newcommand{\cS}{\mathcal{S}}
\newcommand{\cT}{\mathcal{T}}

\newcommand{\cV}{\mathcal{V}}

\newcommand{\bone}{\mathbf{1}}
\newcommand{\bonep}{\mathbf{1}_{[p]}}
\newcommand{\KL}{\mathsf{KL}}

\newcommand{\R}{{\rm I}\kern-0.18em{\rm R}}
\newcommand{\h}{{\rm I}\kern-0.18em{\rm H}}
\newcommand{\K}{{\rm I}\kern-0.18em{\rm K}}
\newcommand{\p}{{\rm I}\kern-0.18em{\rm P}}
\newcommand{\E}{{\rm I}\kern-0.18em{\rm E}}
\newcommand{\Z}{{\rm Z}\kern-0.18em{\rm Z}}
\newcommand{\1}{{\rm 1}\kern-0.24em{\rm I}}
\newcommand{\N}{{\rm I}\kern-0.18em{\rm N}}

\definecolor{MIT}{RGB}{163,31,52}

\newcommand{\inp}[1]{\left(#1\right)}

\crefname{theorem}{Theorem}{Theorems}
\crefname{observation}{Observation}{Observations}
\crefname{claim}{Claim}{Claims}
\crefname{condition}{Condition}{Conditions}
\crefname{example}{Example}{Examples}
\crefname{fact}{Fact}{Facts}
\crefname{lemma}{Lemma}{Lemmas}
\crefname{corollary}{Corollary}{Corollaries}
\crefname{definition}{Definition}{Definitions}
\crefname{remark}{Remark}{Remarks}

\DeclareMathOperator{\Tr}{\mathbf{Tr}}
\DeclareMathOperator{\diag}{\mathbf{diag}}

\makeatletter
\newcommand*{\defeq}{\mathrel{\rlap{%
                     \raisebox{0.3ex}{$\m@th\cdot$}}%
                     \raisebox{-0.3ex}{$\m@th\cdot$}}%
                    =}
\newcommand*{\eqdef}{=
  \mathrel{\rlap{%
      \raisebox{0.3ex}{$\m@th\cdot$}}%
    \raisebox{-0.3ex}{$\m@th\cdot$}}%
}
\makeatother
\usepackage{comment}

\usepackage{fullpage}
\usepackage{geometry}

\begin{document}

\title{Exact recovery in the Ising Blockmodel}

\author{Quentin Berthet\thanks{Email:
    \texttt{\scriptsize q.berthet@statslab.cam.ac.uk}.  %
    University of Cambridge, UK.  Supported by an Isaac Newton Trust
    Early Career Support Scheme and by The Alan Turing Institute under
    the EPSRC grant EP/N510129/1.} \and %
  Philippe Rigollet\thanks{Email: \texttt{\scriptsize rigollet@math.mit.edu}.  %
    Massachusetts Institute of Technology, Cambridge, MA, USA.
    Supported by Grants NSF-DMS-1541099, NSF-DMS-1541100,
    DARPA-BAA-16-46 and a grant from the MIT NEC Corporation.} \and
  Piyush Srivastava\thanks{Email:
    \texttt{\scriptsize piyush.srivastava@tifr.res.in}. %
    Tata Institute of Fundamental Research, Mumbai, MH, India.  This
    work was performed while this author was supported by the United
    States NSF grant CCF-1319745 and was at the Center for the
    Mathematics of Information at the California Institute of
    Technology.}}

\date{}
\maketitle
\thispagestyle{empty}

\begin{abstract}\ We consider the problem associated to recovering the block structure of an Ising model
given independent observations on the binary hypercube. This new model, called the Ising blockmodel, is a perturbation of the mean field approximation of the Ising model known as the Curie--Weiss model: the sites are partitioned into two blocks of equal size and the
interaction between those of the same block is stronger than across blocks, to account for more order within each block. We study probabilistic, statistical and computational aspects of
this model in the high-dimensional case when the number of sites may be much larger
than the sample size.
\end{abstract}



\newpage
\pagestyle{plain}
\setcounter{page}{1}

\section{Introduction}  \label{sec:intro}

The past decades have witnessed an explosion of the amount of data collected. Along with this expansion comes the promise of a better understanding of an observed phenomenon by extracting relevant information from this data. Larger datasets not only call for faster methods to process them but also lead us to completely rethink the way data should be modeled. Specifically, these new datasets arise as the agglomeration of a multitude of basic entities and, rather than their average behavior, most of the information is contained in their interactions.  \emph{Graphical models} (a.k.a Markov Random Fields) have proved to be a very useful tool to turn raw data into networks that are amenable to clustering or community detection. Specifically, given random variables $\sigma_1, \ldots, \sigma_p$, the goal is to output a graph on $p$ nodes, one for each variable, where the edges encode conditional independence between said variables \citep{Lau96}. Graphical models have been successfully employed in a variety of applications such as  image analysis~\citep{Bes86}, natural language processing \citep{ManSch99} and genetics~\citep{LauShe03,SebRamNol05} for example. 

Originally introduced in the context of statistical physics to explain the observed behavior of various magnetic materials~\citep{Isi25}, the Ising Model is a graphical model for binary random variables $\sigma_1, \ldots, \sigma_p \in \{-1, 1\}$, hereafter called \emph{spins}. Despite its simplicity, this model has been effective at capturing a large class of physical systems. More recently, this model was proposed to model social interactions such as political affinities, where $\sigma_j$ may represent the vote of U.S. senator $j$ on a random bill in~\cite{BanEl-dAs08} (see also the data used in~\cite{DiaGoeHol08} for the U.S. House of Representatives). In this context, much effort has been devoted to estimating the underlying structure of the graphical model~\citep{BreMosSly08, RavWaiLaf10, Bre15} under sparsity assumptions. At the same time, the theoretical side of social network analysis has witnessed a lot of activity around the estimation and reconstruction of stochastic blockmodels~\citep{HolLasLei83} as a simple but efficient way to capture the notion of \emph{communities} in social networks. These random graph models assume an underlying partition of the nodes, leading to inhomogeneous connection probabilities between nodes. Given the realization of such a graph, the goal is to recover the partition of the nodes. Already in the context of a balanced partition into two communities, this model has revealed interesting threshold phenomena~\citep{MosNeeSly15, MosNeeSly13, Mas14}. 

In this work, we combine the notions of stochastic blockmodel and that of graphical model by assuming that we observe independent copies of a vector $\sigma=(\sigma_1, \ldots, \sigma_p) \in \{-1,1\}^p$ distributed according to an Ising model with a block structure analogous to the one arising in the stochastic blockmodel.

Specifically, assume that $p\ge 2$ is an even integer and let $S \subset [p]:=\{1, \ldots, p\}$ be a subset of size $|S|=m=p/2$.  For any partition $(S,\bar S)$, where $\bar S=[p]\setminus S$ denotes the complement of $S$, write $i \sim j$ if $(i,j)  \in S^2 \cup \bar S^2$ and $i \nsim j$ if $(i,j) \in [p]^2\setminus (S^2 \cup \bar S^2)$. Fix $\beta,\alpha \in \R$ and let $\sigma \in \{-1, 1\}^p$ have density $f_{S,\alpha, \beta}$ with respect to the counting measure on $\{-1,1\}^p$ given by
\begin{equation}
\label{EQ:ising}
f_{S,\alpha,\beta}(\sigma)=\frac{1}{Z_{\alpha, \beta}}\exp\Big[\frac{\beta}{2p} \sum_{i\sim j} \sigma_i \sigma_j +\frac{\alpha}{2p} \sum_{i \nsim j} \sigma_i \sigma_j \Big]\,,
\end{equation}
where
\begin{equation}
\label{EQ:partition}
Z_{S,\alpha, \beta}:=\sum_{\sigma \in \{-1,1\}^p}\exp\Big[\frac{\beta}{2p} \sum_{i\sim j} \sigma_i \sigma_j +\frac{\alpha}{2p} \sum_{i \nsim j} \sigma_i \sigma_j \Big]\,
\end{equation}
is a normalizing constant traditionally called \emph{partition function}. Let $\p_{S,\alpha, \beta}$ denote the probability distribution over $\{-1,1\}^p$ that has density $f_{S,\alpha, \beta}$ with respect to the counting measure on $\{-1,1\}^p$. We call this model the \emph{Ising Blockmodel} (IBM). We write simply $f_{\alpha, \beta}$ and $\p_{\alpha,\beta}$ to emphasize the dependency on $\alpha,\beta$ and simply $\p_S$ to emphasize the dependency on $S$.

When $\alpha=\beta>0$, the model~\eqref{EQ:ising} is the mean field approximation of the (ferromagnetic) Ising model and is called the \emph{Curie-Weiss} model (without external field). It can be readily seen from~\eqref{EQ:ising} that vectors $\sigma \in \{-1,1\}^p$ that present a lot of pairs $(i,j)$ with opposite spins (high energy configurations), i.e., $\sigma_i\sigma_j<0$, receive less probability than vectors where most of the spins agree (low energy configurations). There are however much fewer vectors with low energy in the discrete hypercube and this tension between \emph{energy} and \emph{entropy} is responsible for phase transitions in such systems.

When positive, the parameter $\beta>0$ is called \emph{inverse temperature} and it controls the strength of interactions, and therefore, the weight given to the energy term. When $\beta  \to 0$, the entropy term dominates and $\p_{\beta,\beta}$ tends to the uniform density over $\{-1, 1\}^p$. When $\beta \to \infty$, $\p_{\beta, 0}\to .5\delta_{\bone}+.5\delta_{-\bone}$, where $\delta_x$ denotes the Dirac point mass at $x$ and $\bone=(1, \ldots, 1) \in \{-1,1\}^p$ denotes the all-ones vector of dimension $p$, the energy term dominates and it affects the global behavior of the system as follows. 

Let $\muCW=\sigma^\top \bone/p$ denote the \emph{magnetization} of $\sigma$. When $\muCW\simeq 0$, then $\sigma$ has a balanced numbers of positive and negative spins (paramagnetic behavior) and when $|\muCW| \gg 0$, then $\sigma$ has a large proportion of spins with a given sign (ferromagnetic behavior).  When $p$ is large enough, the Curie-Weiss model is known to obey a phase-transition from ferromagnetic to paramagnetic behavior when the temperature crosses a threshold (see subsection~\ref{SUB:CW} for details). 
This striking result indicates that when the temperature decreases ($\beta$ increases), the model changes from that of a disordered system (no preferred inclination towards $-1$ or $+1$) to that of an ordered system (a majority of the  spins agree to the same sign). This behavior is interesting in the context of modeling social interactions and indicates that if the strength of interactions is large enough ($\beta>1$) then a partial consensus may be found.
Formally, the Curie-Weiss model may also be defined in the anti-ferromagnetic case $\beta<0$---we abusively call it ``inverse temperature" in this case also---to model the fact that negative interactions are encouraged. For such choices of $\beta$, the distribution is concentrated around  balanced configurations $\sigma$  that have magnetization close to 0. Moreover, as $\beta \to -\infty$, $\p_{\beta,  \beta}$ converges to the uniform distribution on configurations with zero magnetization (assuming that $p$ is even so that such configurations exist for simplicity). As a result, the anti-ferromagnic case arises when no consensus may be found and and the spins are evenly split between positive and negative.

In reality though,  a collective behavior may be fragmented into communities and the IBM is meant to reflect this structure. Specifically, since $\beta >\alpha$, the strength $\beta$ of interactions within the blocks $S$ and $\bar S$ is larger than that across blocks $S$ and $\bar S$. As will become clear from our analysis, the case where $\alpha<0$ presents interesting configurations whereby the two blocs $S$ and $\bar S$ have polarized behaviors, that is opposite magnetization in each block. 

The rest of this paper is organized as follows. In Section~\ref{SEC:prob}, we study the probability distributions $\p_{\alpha, \beta}$, for $\alpha <\beta$ and exhibit phase transitions. Next, in Section~\ref{SEC:stat}, we consider the problem of recovering the partition $S,\bar S$ from $n$ iid samples from $\p_{\alpha, \beta}$.

Finally note that  the size $p$ of the system has to be large enough to observe interesting phenomena. In this paper we are also concerned with such high dimensional systems and our results will be valid for large enough $p$, potentially much larger than the number of observations. In particular, we often consider asymptotic statements as $p \to \infty$. However, in the statistical applications of Section~\ref{SEC:stat} we are interested in understanding the scaling of the number of observations as a function of $p$. To that end, we keep track of the first order terms in $p$ and only let higher order terms vanish when convenient.

\section{Probabilistic analysis of the Ising blockmodel}
\label{SEC:prob}

We will see in Section~\ref{SEC:stat} that given $\sigma^{(1)}, \ldots, \sigma^{(n)}$ that are independent copies of $\sigma \sim \p_{\alpha, \beta}$, the sample covariance matrix $\hat \Sigma$ defined by 
\begin{equation}
\label{EQ:defECM}
\hat \Sigma = \frac{1}{n}\sum_{t=1}^n \sigma\bi {\sigma\bi}^\top\,,
\end{equation}
is a sufficient statistic for $S$. From basic concentration results (see Section~\ref{SEC:stat}), it can be shown that this matrix concentrates around the true covariance matrix $\Sigma=\E_{\alpha, \beta} \big[\sigma \sigma^\top\big]$ where $\E_{\alpha, \beta}$ denotes the expectation associated to $\p_{\alpha, \beta}$. Unfortunately, computing $\Sigma$ directly is quite challenging. Instead, we show that when $p$ is large enough, then $\p_{\alpha, \beta}$ is spiked around specific values, which, in turn, give us a handle of quantities of the form $\E_{\alpha, \beta}[\varphi(\sigma)]$ for some test function $\varphi$. 
Beyond our statistical task, we show phase transitions that are interesting from a probabilistic point of view.

\subsection{Free energy}

Let $\HIBM$ denote the \emph{IBM Hamiltonian} (or ``energy") defined on $\{-1, 1\}^p$ by
\begin{equation}
\label{EQ:defH}
\HIBM(\sigma)=-\big(\frac{\beta}{2p} \sum_{i\sim j} \sigma_i \sigma_j +\frac{\alpha}{2p} \sum_{i \nsim j} \sigma_i \sigma_j \big)\,,
\end{equation}
so that
$$
f_{\alpha, \beta}(\sigma)=\frac{e^{-\HIBM(\sigma)}}{Z_{\alpha, \beta}}
$$
Akin to the Curie-Weiss model, the density $f_{\alpha, \beta}$ puts uniform weights on configurations that have the same magnetization structure. To make this statement precise, for any $A \subset [p]$ define $\bone_A \in \{0,1\}^p$ to be the indicator vector of $A$ and let $\mu_A=\sigma^\top \bone_A/|A|$ denote the \emph{local magnetization} of $\sigma$ on the set $A$. It follows from elementary computations that
\begin{equation}
\label{EQ:defH2}
\HIBM(\sigma)=-\frac{m}{4}\Big(2\alpha \mu_S\mu_{\bar S} + \beta  (\mu_S^2 +\mu_{\bar S}^2  )\Big)\,,
\end{equation}
where we recall that $m=p/2$.  Moreover, the number of configurations $\sigma$ with local magnetizations   $\mu=(\mu_S, \mu_{\bar S}) \in [-1,1]^2$ is given by 
$$
{m \choose \frac{\mu_S+1}{2}m}{m \choose \frac{\mu_{\bar S}+1}{2}m}
$$
This quantity can be approximated using Stirling's formula (see  \Cref{lem:binom-stirling}): For any $\mu \in (1+\eps,1-\eps)$, there exists two positive constants $\ubar{c}, \bar c$ such that  
$$
\frac{\ubar{c}}{\sqrt{m}}e^{-mh\big(\frac{\mu+1}{2}\big)} \le {m \choose \frac{\mu+1}{2}m}\le \frac{\bar{c}}{\sqrt{m}}e^{mh\big(\frac{\mu+1}{2}\big)}\,, \qquad \forall \, m \ge 1
$$
where $h: [0,1] \to \R$ is the binary entropy function defined by $h(0)=h(1)=1$ and for any $s \in (0,1)$ by
$$
h(s)=-s\log(s) - (1-s)\log(1-s)\,.
$$
Thus, IBM induces a marginal distribution on the local magnetizations that has density 
\begin{equation}
\label{EQ:densitymag}
\frac{\ell_m}{mZ_{\alpha, \beta}}\exp\Big[-\frac{m}{4}g(\mu_S, \mu_{\bar S})  \Big]\,,
\end{equation}
where $\ubar{c}^2\le \ell_m \le \bar{c}^2$ and
\begin{equation}
\label{EQ:defG}
 g(\mu_S, \mu_{\bar S})=-2\alpha \mu_S\mu_{\bar S} - \beta (\mu_S^2 +\mu_{\bar S}^2 )-4h\big(\frac{\mu_S+1}{2}\big)-4h\big(\frac{\mu_{\bar S}+1}{2}\big)\,.
\end{equation}
Note that the support of this density is implicitly the set of possible values for pairs local magnetizations of vectors in $\dhyp$, that is the set $\cM^2$, where
\begin{equation}
\label{EQ:defM}
\cM^2:=\big\{\frac{s^\top\bone_{[m]}}{m}, s \in \{-1,1\}^m\big\} \subset [-1,1]^2\,.
\end{equation}
We call the function $g$  the \emph{free energy} of the Ising blockmodel and its structure of minima is known to control the behavior of the system. Indeed,  $g^*$ denote the minimum value of $g$ over $\cM^2$. It follows from~\eqref{EQ:densitymag} that any local magnetization $(\mu_S, \mu_{\bar S})\in \cM^2$ such that $g(\mu_S, \mu_{\bar S}) > g^*$ has a probability exponentially smaller than any magnetization that minimizes $g$ over $\cM^2$. Intuitively, this results in a distribution that is concentrated around its modes. Before quantifying this effect, we study the minima, known as \emph{ground states} of the free energy  $g$, when defined over the continuum $[-1,1]^2$. %

\subsection{Ground states}
\label{sub:ground}

Recall that when $\alpha=\beta$, the block structure vanishes and the IBM reduces to the well-known Curie-Weiss model. We gather in \Cref{SUB:CW} useful facts about the Curie-Weiss model that we use in the rest of this section.

The following proposition characterizes the ground states of the Ising blockmodel. For any $p \in [1, \infty]$, we denote by $\|\cdot\|_p$  the $\ell_p$ norm of $\R^2$ and by $\cB_p=\{x \in \R^2,:\, \|x\|_p\le 1\}$ the unit ball with respect to that norm.

\begin{proposition}
\label{LEM:study_g}
For any $b \in \R$, let $\pm \tilde{x}(b) \in (-1,1), \tilde{x}(b)\ge 0$ denote the ground state(s) of the Curie-Weiss model with inverse temperature $b$.
The free energy $g_{\alpha, \beta}$ of the IBM defined in~\eqref{EQ:defG} has the following  minima:

\noindent If $\beta+|\alpha| \le 2$, then $g_{\alpha, \beta}$ has a unique  minimum at $(0,0)$. \\
\noindent If $\beta +|\alpha|> 2$, then  three cases arise:
\begin{enumerate}
\item If $\alpha=0$, then $g_{\alpha, \beta}$ has four  minima at $(\pm \tilde{x}(\beta/2),\pm \tilde{x}(\beta/2))$,
\item If $\alpha>0$, $g_{\alpha, \beta}$ has two  minima at $\tilde s	= ( \tilde{x}(\frac{\beta+\alpha}{2}), \tilde{x}(\frac{\beta+\alpha}{2}))$ and $-\tilde s$,
\item If $\alpha<0$, $g_{\alpha, \beta}$ has two  minima at $\tilde s = ( \tilde{x}(\frac{\beta-\alpha}{2}), -\tilde{x}(\frac{\beta-\alpha}{2}))$ and $-\tilde s$.
\end{enumerate}
In particular, for all  values of the parameters $\alpha$ and $\beta$, all ground states $(\tilde x, \tilde y)$ satisfy $\tilde x^2=\tilde y^2<1$.

\end{proposition}
This result is illustrated in Figure~\ref{FIG:contour}, composed of contour plots of the free energy $g_{\alpha,\beta}$ on the square $[-1,1]^2$, for several values of the parameters. The different regions are also represented in Figure~\ref{FIG:phases} below.

\begin{figure}[h!]
    \begin{center} 
    \begin{tabular}{cccc}
    	\includegraphics[width=.22 \textwidth]{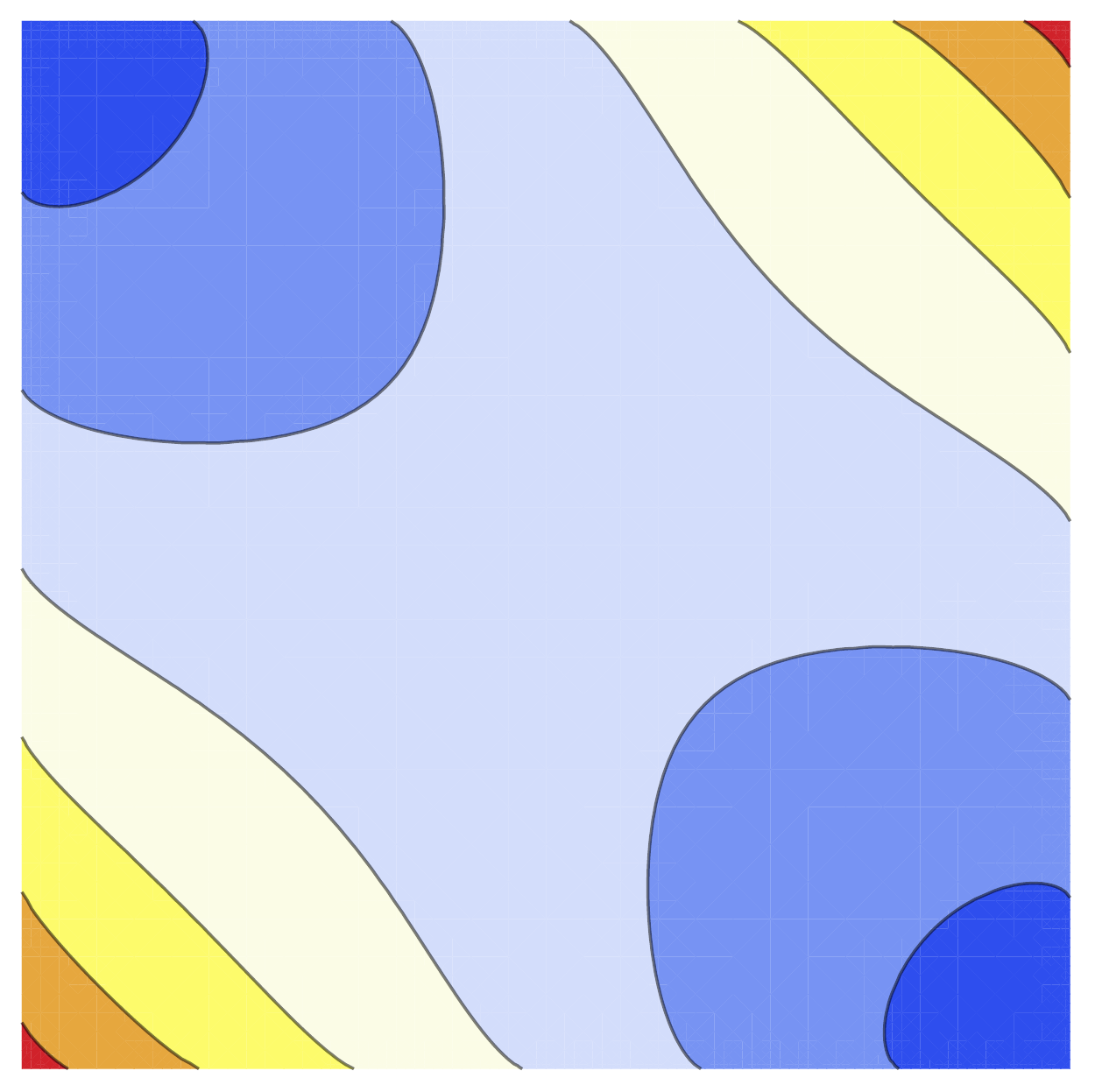}&
    	\includegraphics[width=.22 \textwidth]{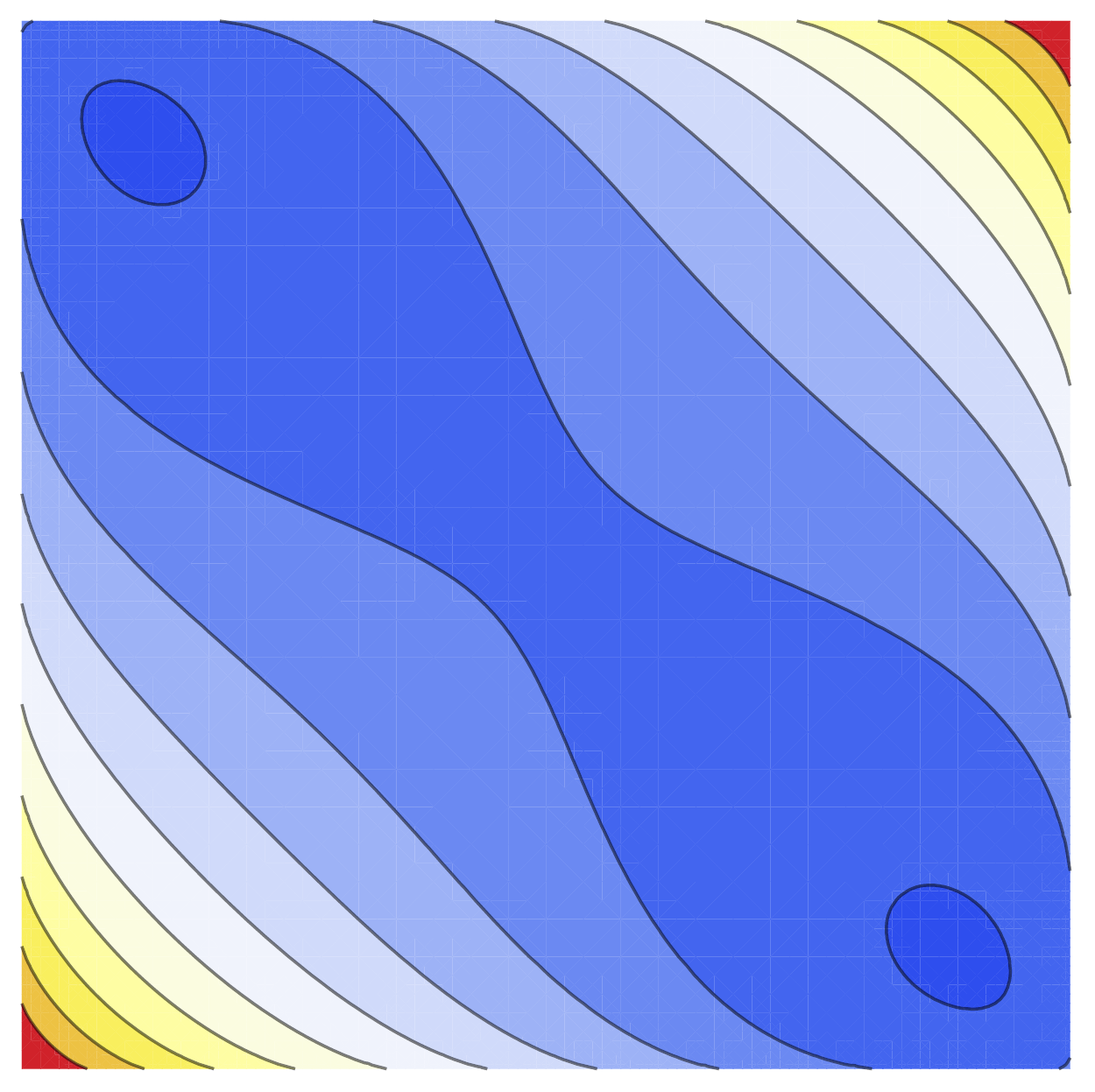}&      
	\includegraphics[width=.22 \textwidth]{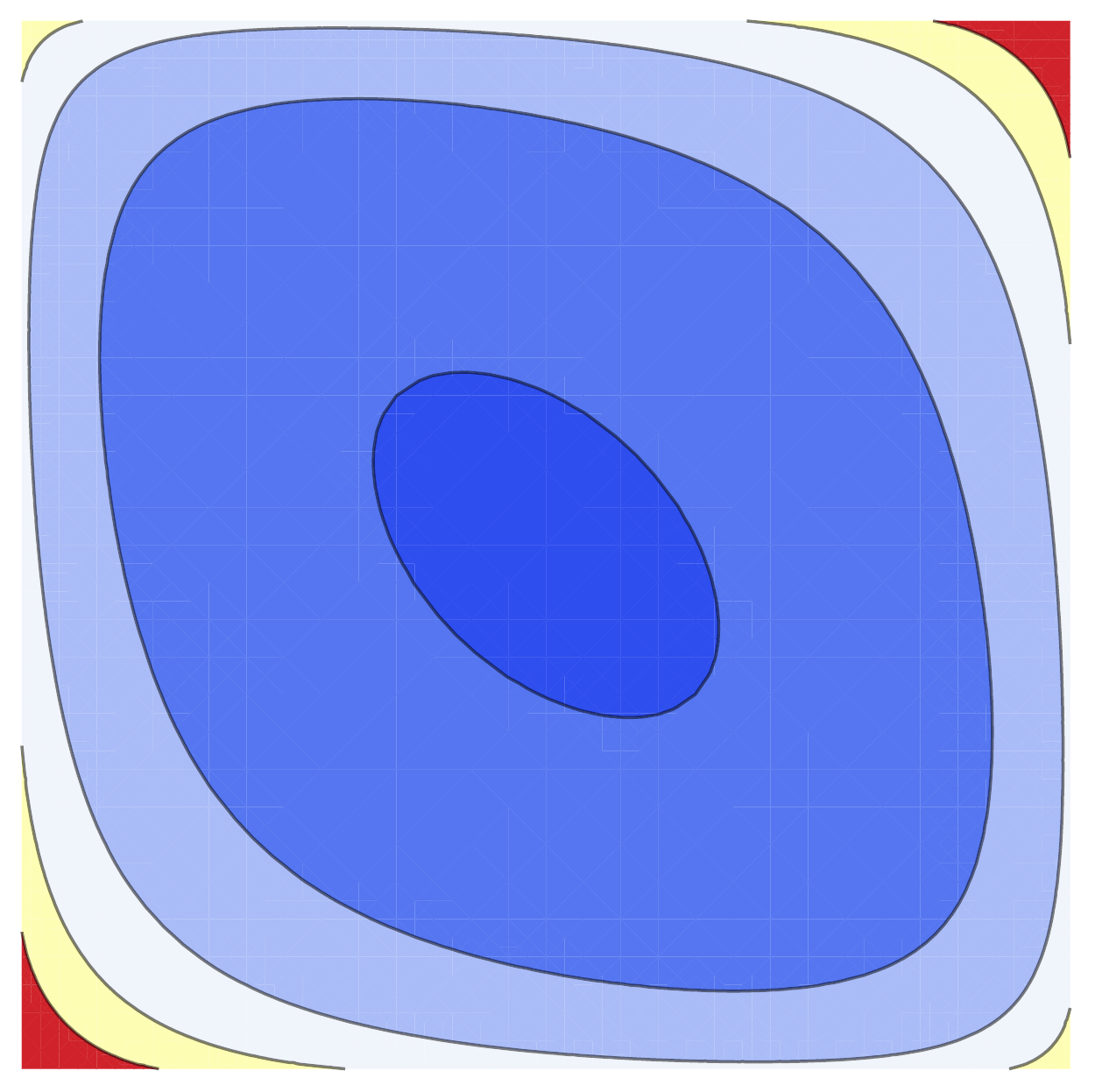}&
	\includegraphics[width=.22 \textwidth]{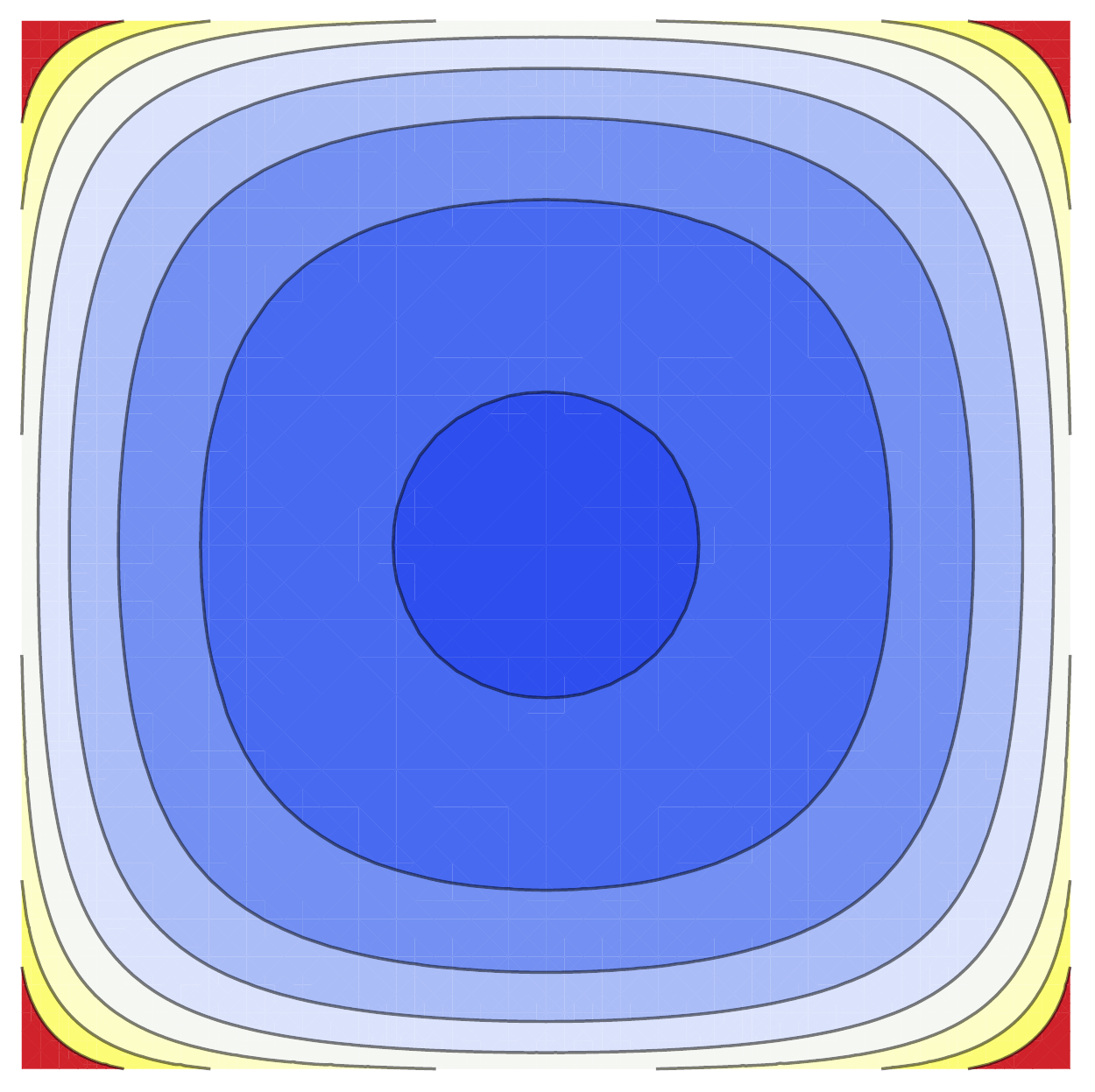}   \\
	$\alpha=-6$&$\alpha=-2.5$&$\alpha=-0.5$&$\alpha=0$\\
	\phantom{bouh!}\\
	\includegraphics[width=.22 \textwidth]{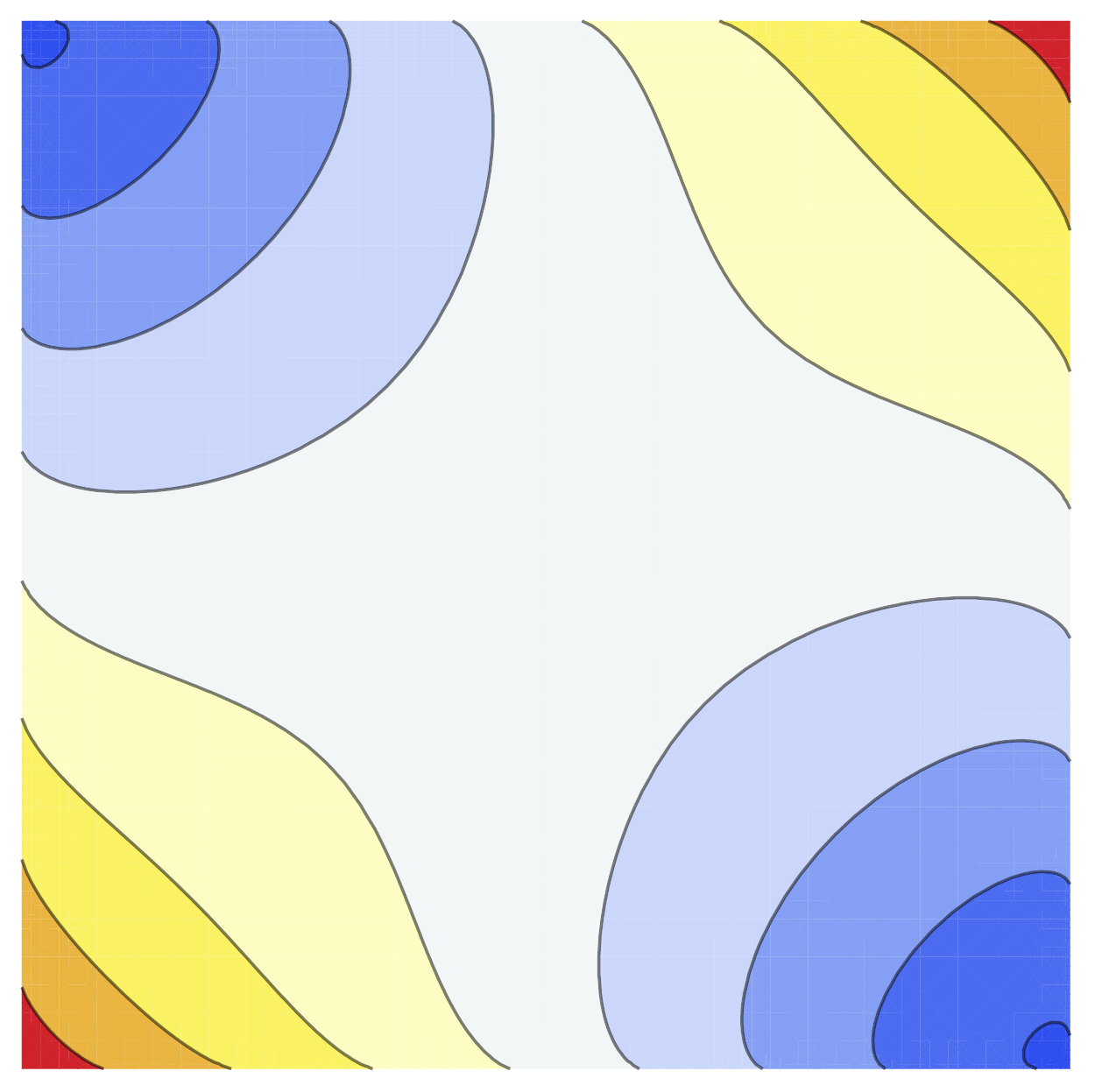}&
    	\includegraphics[width=.22 \textwidth]{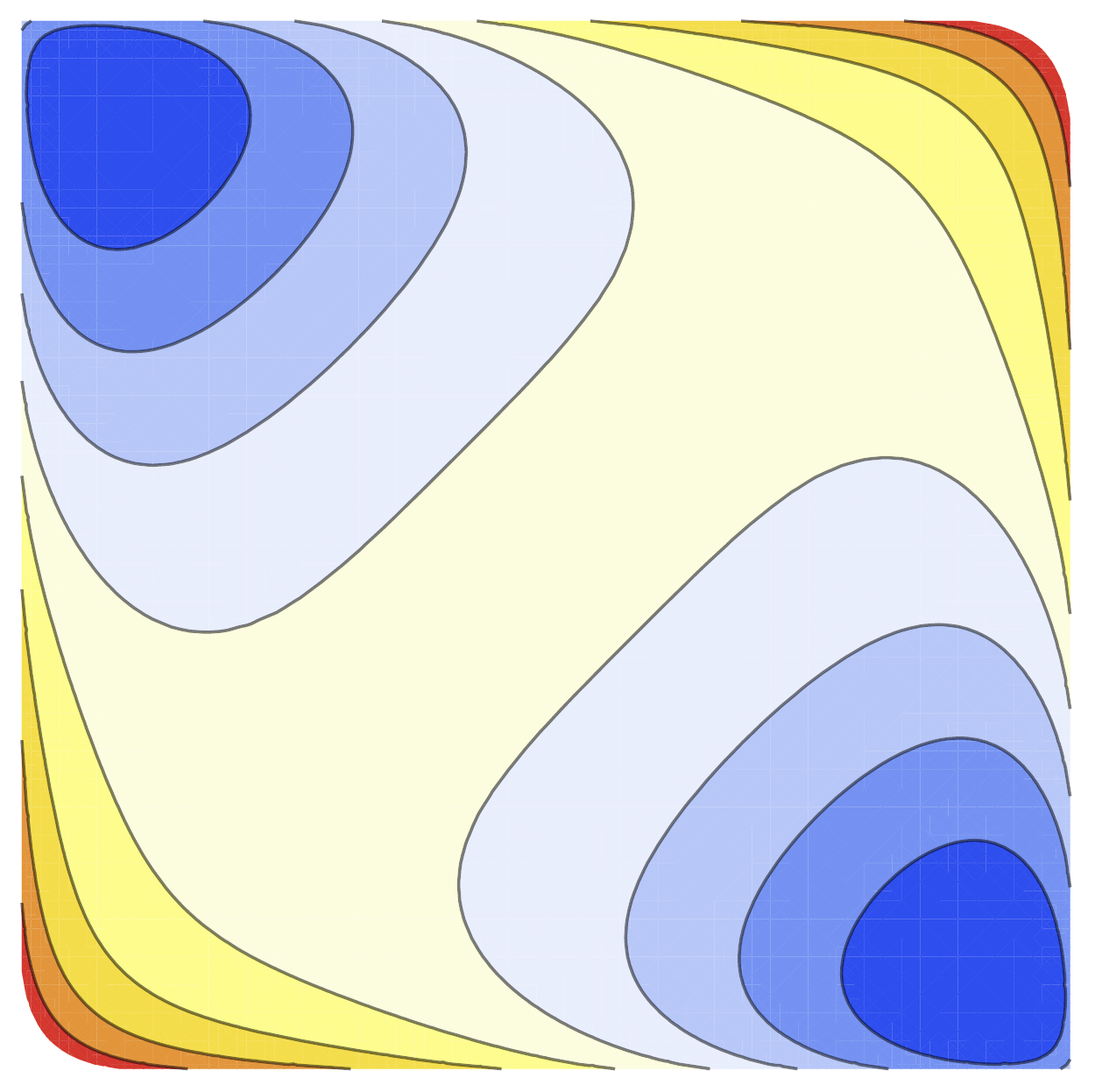}&      
	\includegraphics[width=.22 \textwidth]{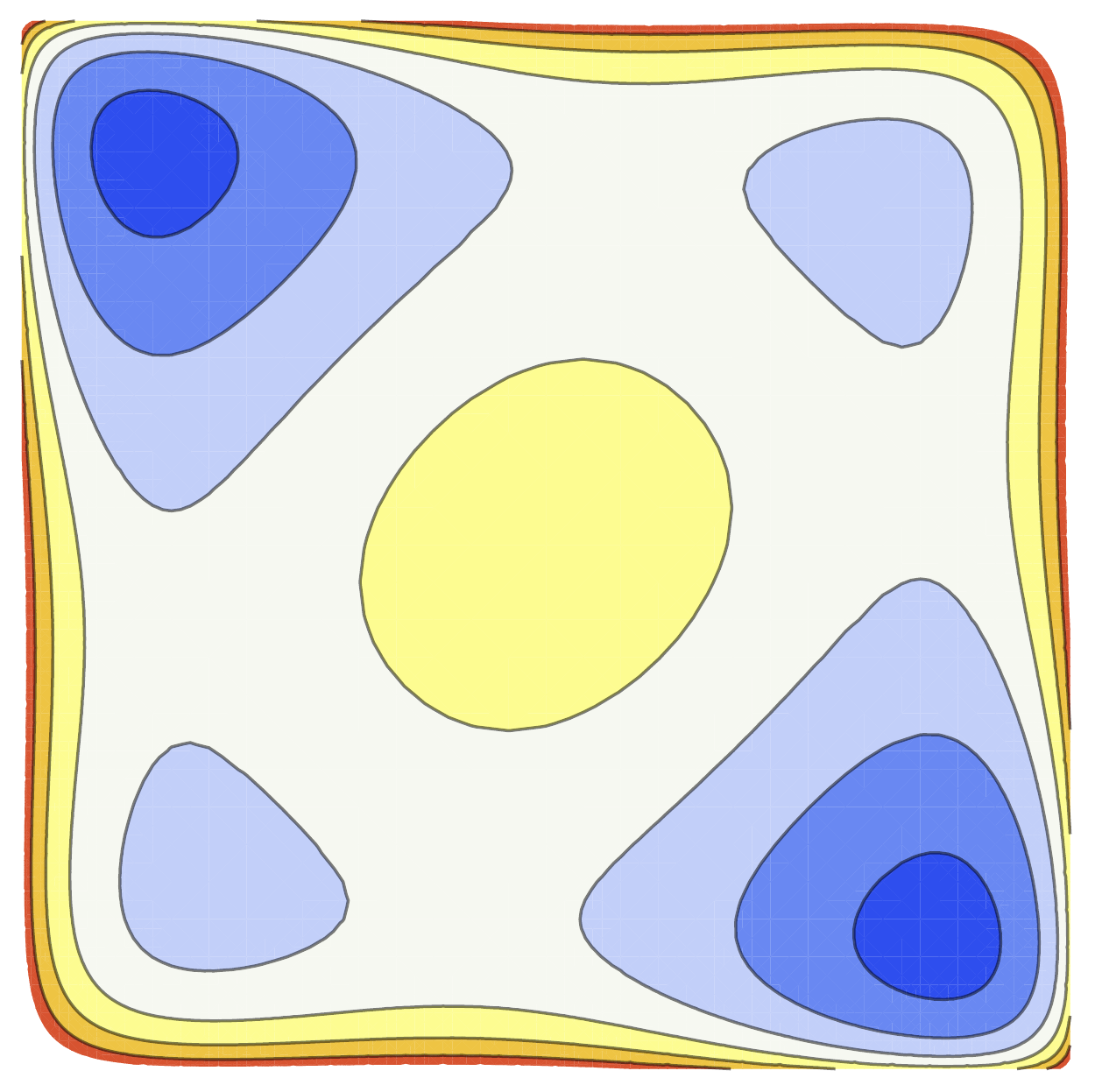}&
	\includegraphics[width=.22 \textwidth]{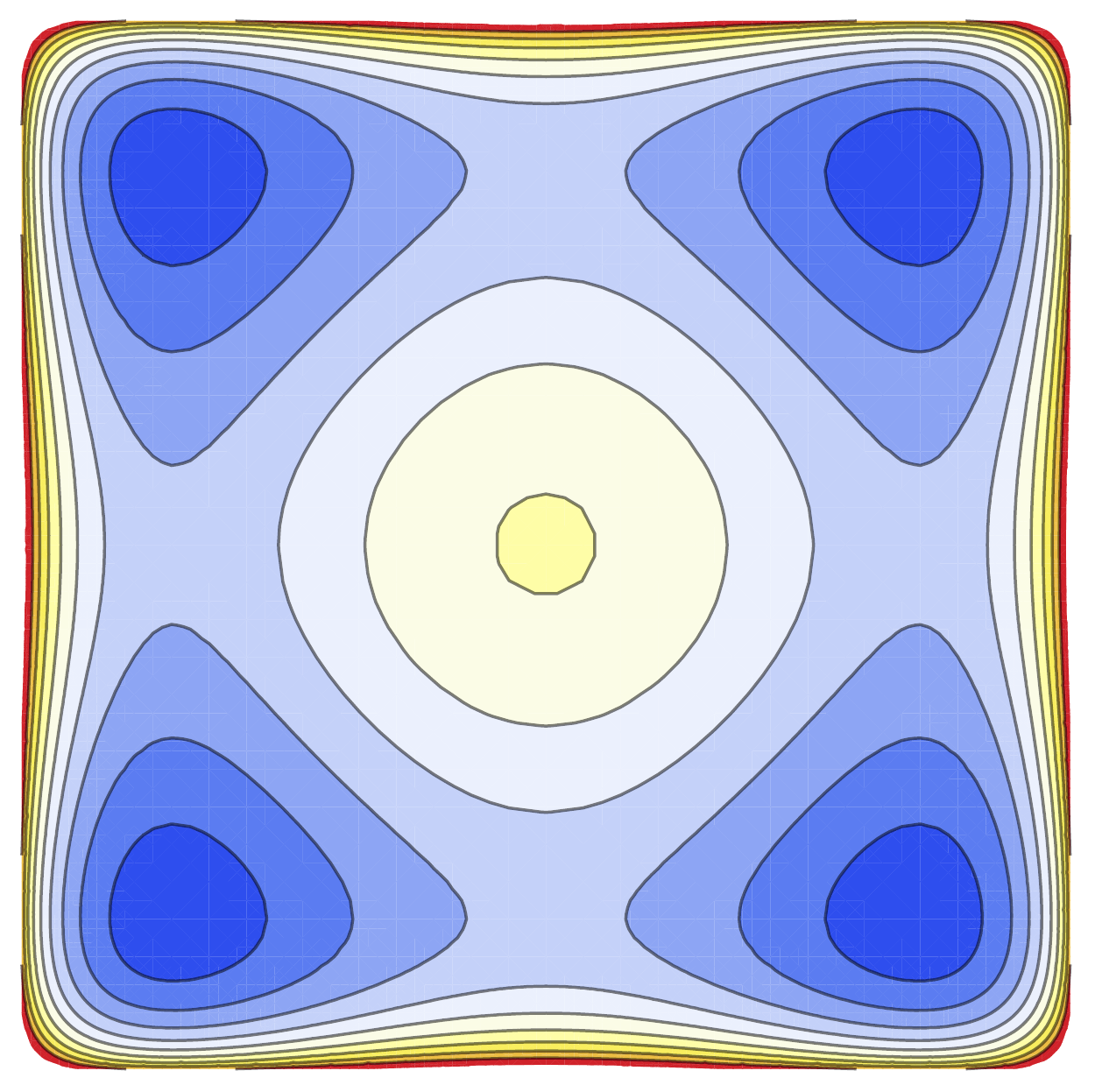}   \\
	$\alpha=-4$&$\alpha=-0.9$&$\alpha=-0.2$&$\alpha=0$
    \end{tabular}
    \end{center} 
    \caption{Contour plots of the values of the free energy $g_{\alpha,\beta}$ with higher values in red and lower values in blue, corresponding to ground states. {\em Top row} : Several choices for $\alpha<0$, and $\beta=1.5<2$. {\em Bottom row} : Several choices for $\alpha<0$, and $\beta=2.5>2$. The same plots with $\alpha>0$ can be obtained by a $90^\circ$ rotation, by symmetry of the function.}
    \label{FIG:contour}
\end{figure}

\begin{proof}
Throughout this proof, for any $b \in \R$, we denote by $\gCW_{b}(x), x \in [-1,1]$, the free energy of the Curie-Weiss model with inverse temperature $b$. We write $g:=g_{\alpha, \beta}$ for simplicity to denote the free energy of the IBM.%

Note that
\begin{equation}
\label{EQ:freeNRJ}
g(x,y)=\gCW_{\frac{\beta+\alpha}{2}}(x) + \gCW_{\frac{\beta+\alpha}{2}}(y)+\alpha(x-y)^2\,.
\end{equation}
We split our analysis according to the sign of $\alpha$. Note first that if $\alpha=0$, we have
$$
g(x,y)=\gCW_{\frac{\beta}{2}}(x) + \gCW_{\frac{\beta}{2}}(y)\,.
$$
It yields that:
\begin{itemize}[leftmargin=*]
\item  If $\beta \le 2$, then $\gCW_{\frac{\beta}{2}}$ has a unique local minimum at $x=0$ which implies that $g$ has a unique minimum at $(0,0)$
\item If $\beta > 2$, then  $\gCW_{\frac{\beta}{2}}$ has exactly two   minima at $\tilde{x}(\beta/2)$ and $-\tilde{x}(\beta/2)$, where $\tilde{x}(\beta/2) \in (-1,1)$. It implies that $g$ has four  minima at $(\pm \tilde{x}(\beta/2),\pm \tilde{x}(\beta/2))$.
\end{itemize}

Next, if $\alpha > 0$, in view of~\eqref{EQ:freeNRJ}
we have
$$
g(x,y)\ge \gCW_{\frac{\beta+\alpha}{2}}(x) + \gCW_{\frac{\beta+\alpha}{2}}(y)
$$
with equality iff $x=y$. It follows that:
\begin{itemize}[leftmargin=*]
\item  If $\alpha+\beta \le 2$, then  $g$ has a unique minimum at $(0,0)$
\item If $\alpha + \beta > 2$, then $g$ has two  minima on $\cA$ at $( \tilde{x}(\frac{\beta+\alpha}{2}), \tilde{x}(\frac{\beta+\alpha}{2}))$ and at $( -\tilde{x}(\frac{\beta+\alpha}{2}), -\tilde{x}(\frac{\beta+\alpha}{2}))$.
\end{itemize}

Finally, note that $(x-y)^2\le 2x^2 + 2y^2$ with equality iff $x=-y$. Thus, if $\alpha<0$, in view of~\eqref{EQ:freeNRJ} we have
\begin{equation}
\label{EQ:freeNRJalphaneg}
g(x,y)\ge \gCW_{\frac{\beta-\alpha}{2}}(x) + \gCW_{\frac{\beta-\alpha}{2}}(y)
\end{equation}

with equality iff $x=-y$. It implies that
\begin{itemize}[leftmargin=*]
\item If $\beta-\alpha \le 2$, then  $g$ has a unique minimum at $(0,0)$
\item If $\beta-\alpha > 2$, then $g$ has two  minima  at $( \tilde{x}(\frac{\beta-\alpha}{2}), -\tilde{x}(\frac{\beta-\alpha}{2}))$ and at\\ $( -\tilde{x}(\frac{\beta-\alpha}{2}), \tilde{x}(\frac{\beta-\alpha}{2}))$.
\end{itemize}
\end{proof}
\vskip -1cm
\begin{figure}[hp]
\includegraphics[width=\textwidth]{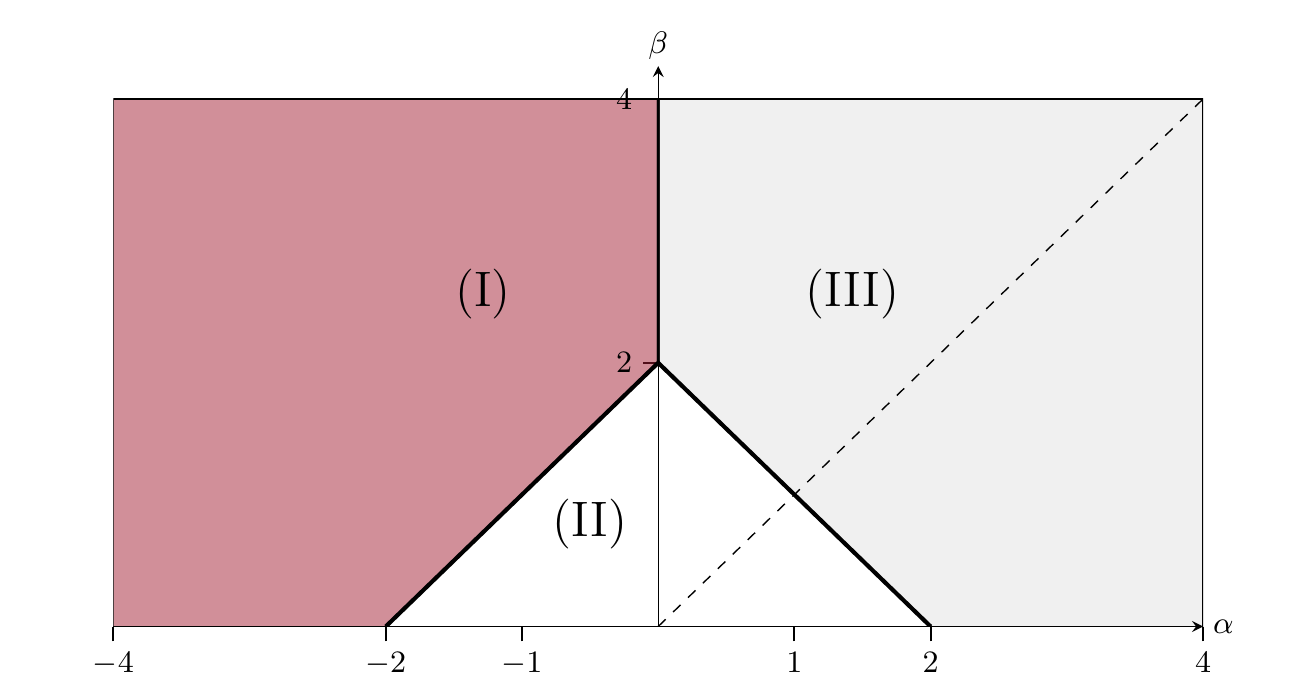}
\caption{Phase diagram of the Ising block model, with three regions for $\alpha$ and $\beta>0$. In region (I), where $\alpha<0$ and $\beta+|\alpha|>2$, there are two ground states of the form $(x, -x)$ and $(-x,x)$. In region (II), where $\beta+|\alpha|<2$, there is one ground state at $(0,0)$. In region (III), where $\alpha>0$ and $\beta+|\alpha|>2$, there are two ground states of the form $(x, x)$ and $(-x,-x)$. %
The dotted line has equation $\alpha=\beta$, we only consider parameters in the region to its left.
}

\label{FIG:phases}
\end{figure}

Using the localization of the ground states from \Cref{LEM:CW}, we also get the following  local and global behaviors of the free energy of the IBM around the ground states.
\begin{lemma}
\label{LEM:study_g2}
Assume that $\beta + |\alpha|\neq 2$. Denote by $(\tilde x, \tilde y)$ any ground state of Ising blockmodel and recall that $\tilde x^2=\tilde y^2$. Then the following holds:
\begin{enumerate}
\item The Hessian $H_{\alpha, \beta}$ of $g_{\alpha, \beta}$ at $(\tilde x, \tilde y)$ is given by
$$
H_{\alpha, \beta}=-2\left(\begin{array}{cc}\beta & \alpha \\\alpha & \beta \end{array}\right) + \frac{4}{1-\tilde x^2} I_2\,.
$$
In particular $H_{\alpha, \beta}$ has eigenvalues $2(\alpha - \beta) +4/(1-\tilde x^2)$ and
 
 $-2(\alpha + \beta) +4/(1-\tilde x^2)$ associated with eigenvectors $(1,-1)$ and $(1,1)$ respectively.\\

\item There exists positive constants $\delta=\delta(\beta+|\alpha|)$, $\kappa^2=\kappa^2(\beta+|\alpha|)$ such that the following holds. For any $(x,y) \in (-1,1)^2$, we have
\begin{equation}
\label{EQ:quadrlb}
g(x,y) \ge g(\tilde x, \tilde y) + \frac{\kappa^2}{2}\Big( \|(x,y)-(\tilde x, \tilde y)\|_\infty\wedge \delta\Big)^2\,.\end{equation}
Moreover, 
 
\noindent If $\beta+|\alpha|>2$, we can take $\DS \delta=e^{-(\beta+|\alpha|)}\frac{\beta + |\alpha|-2}{4(\beta + |\alpha|)}$ and 

$\DS  \kappa^2=1-\frac2{\beta +|\alpha|}$.

\noindent If $\beta+|\alpha|<2$, we can take $\DS \delta=\sqrt{(2-(\beta+|\alpha|))/6}$ and 

$\DS  \kappa^2=2-(\beta+|\alpha|)$.

\end{enumerate}
\end{lemma}
\begin{proof}
Elementary calculus yields directly that
$$
H_{\alpha, \beta}=\left(\begin{array}{cc}-2\beta+\frac{4}{1-\tilde x^2} & -2\alpha \\-2\alpha & -2\beta+\frac{4}{1-\tilde y^2} \end{array}\right)\,.
$$
Moreover, it follows from \Cref{LEM:study_g} that all ground states satisfy $\tilde x^2=\tilde y^2$. This completes the proof of the first point.

 We now turn to the proof of the second point and split the analysis into four cases: (i) $\alpha\ge 0$ and  $\beta + \alpha<2$, (ii)   $\alpha\ge 0$ and  $\beta + \alpha>2$, (iii)  $\alpha<0$ and  $\beta - \alpha<2$, (iv)  $\alpha<0$ and  $\beta - \alpha>2$.

\medskip
 
\noindent {\it Case (i): $\alpha>0$ and  $\beta + \alpha<2$.} Recall that in this case, $g$ has a unique minimum at $(0,0)$. Therefore, in view of~\eqref{EQ:freeNRJ} and  \Cref{LEM:CW}, we have
 \begin{align*}
g(x,y)-g(0,0)&= \gCW_\frac{\beta+|\alpha|}{2} (x)-\gCW_\frac{\beta+|\alpha|}{2} (0) +\gCW_\frac{\beta+|\alpha|}{2} (y)-\gCW_\frac{\beta+|\alpha|}{2} (0)+\alpha(x-y)^2 \\
&\ge \frac12\big(2-(\beta + |\alpha|)\big)\big[(|x-0|\wedge \eps')^2+ (|y-0|\wedge \eps')^2\big]\\
&\ge \frac12\big(2-(\beta + |\alpha|)\big)\big(\|(x,y) - (0,0)\|_\infty\wedge \eps'\big)^2\,.
\end{align*}
where $\DS \eps'=\sqrt{(2-(\beta+|\alpha|))/6}$ which concludes this case.

\medskip

\noindent {\it Case (ii): $\alpha>0$ and  $\beta + \alpha>2$.} Recall that in this case, $g$ has two  minima denoted generically by  $(\tilde x, \tilde y)$ where $\tilde x=\tilde y$. Therefore, in view of~\eqref{EQ:freeNRJ} and \Cref{LEM:CW}, we have
 \begin{align*}
g(x,y)-g(\tilde x,\tilde y)&= \gCW_\frac{\beta+|\alpha|}{2} (x)-\gCW_\frac{\beta+|\alpha|}{2} (\tilde x) +\gCW_\frac{\beta+|\alpha|}{2} (y)-\gCW_\frac{\beta+|\alpha|}{2} (\tilde y)+\alpha(x-y)^2 \\
&\ge \frac12\big(1-\frac2{\beta + |\alpha|}\big)\big[(|x-0|\wedge \eps)^2+ (|y-0|\wedge \eps)^2\big]\\
&\ge \frac12\big(1-\frac2{\beta + |\alpha|}\big)\big(\|(x,y) - (0,0)\|_\infty\wedge \eps\big)^2\,.
\end{align*}
where  $\DS  \eps=e^{-(\beta+|\alpha|)}\frac{\beta + |\alpha|-2}{4(\beta + |\alpha|)}$ which concludes this case.

\medskip

\noindent {\it Case (iii): $\alpha<0$ and  $\beta - \alpha<2$.} Recall that in this case, $g$ has a unique  minimum at $(0,0)$. Moreover, in view of~\eqref{EQ:freeNRJalphaneg} and \Cref{LEM:CW}, it holds
 \begin{align*}
g(x,y)-g(0,0)&\ge \gCW_\frac{\beta+|\alpha|}{2} (x)-\gCW_\frac{\beta+\alpha}{2} (0) +\gCW_\frac{\beta+|\alpha|}{2} (y)-\gCW_\frac{\beta+\alpha}{2} (0) \\
&\ge \gCW_\frac{\beta+|\alpha|}{2} (x)-\gCW_\frac{\beta+|\alpha|}{2} (0) +\gCW_\frac{\beta+|\alpha|}{2} (y)-\gCW_\frac{\beta+|\alpha|}{2} (0) \\
&\ge \frac12\big(2-(\beta + |\alpha|)\big)\big(\|(x,y) - (0,0)\|_\infty\wedge \eps'\big)^2\,.
\end{align*}
where in the second inequality, we used the fact that $$\gCW_\frac{\beta+\alpha}{2} (0)=\gCW_\frac{\beta+|\alpha|}{2} (0)=-4h(1/2)\,,$$ and we concluded as in Case (i).

\medskip

\noindent {\it Case (iv): $\alpha<0$ and  $\beta - \alpha>2$.}  Recall that in this case, $g$ has two  minima denoted generically by  $(\tilde x, \tilde y)$ where $\tilde x=-\tilde y$. Therefore, in view of~\eqref{EQ:freeNRJ} and~\eqref{EQ:freeNRJalphaneg}, we have
 \begin{align*}
g(x,y)-g(\tilde x,\tilde y)&\ge \gCW_\frac{\beta+|\alpha|}{2} (x)-\gCW_\frac{\beta+\alpha}{2} (\tilde x) +\gCW_\frac{\beta+|\alpha|}{2} (y)-\gCW_\frac{\beta+\alpha}{2} (-\tilde x) -4\alpha\tilde x^2\,.
\end{align*}
Next, observe that from the definition~\eqref{EQ:freeNRJCW} of the free energy in the Curie-Weiss model, we have
$$
-\gCW_\frac{\beta+\alpha}{2} (\tilde x) -\gCW_\frac{\beta+\alpha}{2} (-\tilde x)-4\alpha\tilde x^2=-\gCW_\frac{\beta+|\alpha|}{2} (\tilde x) -\gCW_\frac{\beta+|\alpha|}{2} (-\tilde x)\,.
$$
The above two displays yield
 \begin{align*}
g(x,y)-g(\tilde x,\tilde y)&\ge \gCW_\frac{\beta+|\alpha|}{2} (x)-\gCW_\frac{\beta+|\alpha|}{2} (\tilde x) +\gCW_\frac{\beta+|\alpha|}{2} (y)-\gCW_\frac{\beta+|\alpha|}{2} (-\tilde x) \\
&\ge \frac12\big(1-\frac2{\beta + |\alpha|}\big)\big(\|(x,y) - (0,0)\|_\infty\wedge \eps\big)^2\,.
\end{align*}
where we concluded as in Case~(ii).

\end{proof}

\subsection{Concentration}
\label{sub:concentration}
As mentioned above, quantities of the form $\E_{\alpha, \beta}[\varphi(\sigma)]$ cannot in general be computed explicitly in the IBM. Fortunately, it will be sufficient for us to compute quantities of the form $\E_{\alpha, \beta}[\varphi(\mu)]$, where we recall that $\mu=(\mu_S, \mu_{\bar S})$ denotes the pair of local magnetizations of a random configuration $\sigma \in \dhyp$ drawn according to $\p_{\alpha,\beta}$.  While exact computation is still a hard problem, these quantities can be be well approximated using the fact that $\p_{\alpha,\beta}$ is highly concentrated around its ground states for large enough $p$.  

To leverage concentration, we need to consider the ``large $m$" (or equivalently ``large $p$") asymptotic framework. As a result, it will be convenient to write for two sequences $a_m, b_m$ that $a_m \simeq_m b_m$ if $a=(1+o_m(1))b_m$. 

Our main result hinges on the following proposition that compares the distribution of $\mu=(\mu_S, \mu_{\bar S}) \in [-1,1]$ to a certain mixture of Gaussians that are centered at the ground states.

\begin{theorem}
\label{lem:unwprob}
Let $\varphi: [-1, 1]^2 \to [0,1]$ be any nonnegative bounded continuous test function. Denote by $\tilde s$ any ground state and assume that there exists positive constants  $C, \gamma$, for which $\E\big[\varphi\big(\tilde s+\frac{2}{\sqrt{m}}H^{-1/2}Z\big)\big] \ge {C}{m^{-\gamma}}$
where $Z \sim \cN_2(0, I_2)$ and $H=H_{\alpha, \beta}$ denotes the Hessian of the free energy $g_{\alpha, \beta}$ at $\tilde s$. 
Then
$$
\E_{\alpha, \beta}[\varphi(\mu)]\simeq_m \frac{1}{|G|}
\sum_{\tilde s \in G}\E\big[\varphi(\tilde s+ \frac{2}{\sqrt{m}}H^{-1/2}Z)\big]\,.
$$
where $G \subset \{(\pm \tilde x, \pm \tilde x)\}$ denotes the set of ground states of the IBM.
\end{theorem}
\begin{proof}
Recall that $\cM^2$ defined in~\eqref{EQ:defM} denotes the set of possible values for pairs of local magnetization and observe that
$$
\E_{\alpha, \beta}[\varphi(\mu)]=\frac{1}{Z_{\alpha,\beta}}\sum_{\mu \in \cM^2}\varphi(\mu) z_m(\mu)\, ,
$$
where
  \begin{equation}
  \label{EQ:defzm}
z_m(\mu):=  
    \exp\inp{
      -\frac{m}{4}\inp{-2\alpha \mu_S\mu_{\bar S} -
        \beta  (\mu_S^2 +\mu_{\bar S}^2  )}}
    {m \choose \frac{\mu_S+1}{2}m}
    {m \choose \frac{\mu_{\bar S}+1}{2}m}
  \end{equation}

We split the local magnetization $\mu$  according to their $\ell_2$ distance to the closest ground state. Let $G \subset [0,1]^2$ denote the set of ground states and fix
  $\delta \defeq (\rho/\kappa)\sqrt{(\log m)/m}$, where $\rho>0$ is a constant to be chosen later and $\kappa$ is defined in \Cref{LEM:study_g2}. For any  ground state $\tilde s \in G$, define $\cV_{\tilde s}$ to be the neighborhood of $\tilde s$ defined by
$$
\cV_{\tilde s}=\big\{ \mu \in \cM^2 \,:\, \|\mu-\tilde s\|_\infty \le \delta\big\}\,,
$$ 
where $\delta>0$ is also defined in \Cref{LEM:study_g2}.
Moreover, define 
$$
\cV = \bigcup_{\tilde s \in G}\cV_{\tilde s}\,,
$$
and assume that $m$ is large enough so that (i) the above union is a disjoint one and (ii), there exists a constant $C>0$ depending on $\alpha$ and $\beta$ such that for any  $(x,y) \in \cV$, we have $||x|-1|\wedge ||y|-1| \ge C>0$. Denote by $g_{\alpha,\beta}^*$ the value of the free energy at any of the ground states.

We first treat the magnetizations outside $\cV$. Using \Cref{LEM:study_g2} together with \Cref{lem:binom-bound}, we get
  \begin{align}
   0\le \exp\big(\frac{m}{4}  g_{\alpha,\beta}^*\big)
    \sum_{\mu \notin \cV} \varphi(\mu)z_m(\mu)   &\leq
    \exp\big(\frac{m}{4} g_{\alpha,\beta}^*\big) \sum_{\mu \notin \cV}
    \exp\inp{-\frac{m}{4} g_{\alpha,\beta}(\mu)}\nonumber\\
    &\leq m^2 \exp\big(-\frac{m}{4}\frac{\kappa^2\delta^2}{2}\big)\le m^{2-\frac{\rho^2}{2}}=o_m(m^{-\gamma})\,,\label{eq:2}
\end{align}
for $\rho>4\sqrt{8\gamma}$. 

Next assume that $\mu \in \cV$. Our starting point is
the following approximation, that follows from \Cref{lem:binom-stirling}: for any $\mu \in \cV$, 
\begin{equation}
\label{eq:3}
z_m(\mu)= \frac{1}{\pi m}
    \frac{\exp\inp{-\frac{m}{4}
        g_{\alpha,\beta}(\mu_S, \mu_{\bar{S}}) }
    }{
      \sqrt{(1-\mu_S^2)(1-\mu_{\bar{S}}^2)}
    }(1+o_m(1))\,,
\end{equation}

Define $\cV'= \cV_{\tilde s}-\{\tilde s\}$.
A Taylor expansion
around $\tilde s$ gives for nay $u \in \cV'$,
$$
  g_{\alpha,\beta}(\tilde s + u) = g_{\alpha, \beta}(\tilde s)
  + \frac{1}{2}u^\top H u + O(\delta^3).
$$
where $H=H_{\alpha, \beta}$ denotes the Hessian of $g_{\alpha,\beta}$ at the ground state $\tilde s$.
The above two displays yield
\begin{align*}
 \exp&\big(\frac{m}{4} g_{\alpha,\beta}^*\big)
    \sum_{\mu \in \cV_{\tilde s}}  \varphi(\mu) z_m(\mu)\\
    &= \exp\big(\frac{m}{4} g_{\alpha,\beta}^*\big)
    \sum_{u \in\cV'} \varphi(\tilde s + u)z_m(\tilde s + u)\\
   &\simeq_m\frac{1}{\pi m(1-\tilde x^2)}\sum_{u \in\cV'}\varphi(\tilde s + u)\exp\big( -\frac{m}{8} u^\top H u \big)\\
&\simeq_m\frac{m}{\pi (1-\tilde x^2)}\int_{\delta \cB_\infty}\varphi(\tilde s +x)\exp\big( -\frac{m}{8} x^\top H x \big)dx\\
&=\frac{1}{\pi (1-\tilde x^2)}\frac{1}{\sqrt{\det H}}\int_{\|H^{-\frac{1}{2}}z\|_\infty\le \frac{\delta\sqrt{m}}{2}}\varphi\big(\tilde s+\frac{2}{\sqrt{m}}H^{-1/2}z\big)\exp\big( -\frac{\|z\|_2}{2} \big)dz\\
&\simeq_m\frac{1}{1-\tilde x^2}\frac{2}{\sqrt{\det H}}\Big(\E\big[\varphi\big(\tilde s+\frac{2}{\sqrt{m}}H^{-1/2}Z\big)\big]- T_m\Big)\,.
\end{align*}
where $Z \sim \cN_2(0,I_2)$ and
$$
T_m=\int_{z \,:\,z^\top H^{-1}z \ge \frac{m\delta^2}{2}  }\varphi\big(\tilde s+\frac{2}{\sqrt{m}}H^{-1/2}z\big)\exp\big( -\frac{\|z\|_2}{2} \big)dz
$$
Here, the third equality replaces the sum by a Riemann integral and in
the last one we use the following facts: (i)  the set of vectors $z$ satisfying $\|H^{-\frac{1}{2}}z\|_\infty\le 1$ contains a Euclidean ball of positive radius $r(\alpha, \beta)$ and (ii) $\delta \sqrt{m}\to \infty$. %
Next, observe that since $\varphi$ takes values in $[0,1]$, we have
\begin{align}
0\le T_m&\le 2\pi \p(Z^\top H Z \ge  m/2)\nonumber\\
&\le 2\pi \p\big(\|Z\|^2-2 \ge  \frac{m}{2 \lambda_{\max}(H)}-2\big)\nonumber\\
&\le 2\pi \sqrt{e}\exp \big(-\frac{m}{8 \lambda_{\max}(H)}  \big)=o(m^{-\gamma}) \label{eq:boundTm}
\end{align}
for $m \ge 8 \lambda_{\max}(H)$ and where we used~\Cref{lem:laumas00}.

Since the same calculation holds for all ground states in $G$, and because the sets $\cV_{\tilde s},\ \tilde s\in G$ are disjoint, we get that
\begin{align*}
 \exp&\big(\frac{m}{4} g_{\alpha,\beta}^*\big)
    \sum_{\mu \in \cV } \varphi(\mu) z_m(\mu)\simeq_m\frac{1}{1-\tilde x^2}\frac{2}{\sqrt{\det H}}\sum_{\tilde s \in G}\E\big[\varphi\big(\tilde s+\frac{2}{\sqrt{m}}H^{-1/2}Z\big)\big]\,.
\end{align*}
Together with~\eqref{eq:2}, the above display yields 
$$
\sum_{\mu \in \cM^2}\varphi(\mu) z_m(\mu)\simeq_m \frac{2e^{-\frac{m}{4} g_{\alpha,\beta}^*}}{(1-\tilde x^2)\sqrt{\det H}}
\sum_{\tilde s \in G}\E\big[\varphi(\tilde s+ \frac{2}{\sqrt{m}}H^{-1/2}Z)\big]\,,
$$
In particular, this expression yields for $\varphi \equiv 1$, 
$$
Z_{\alpha, \beta}\simeq_m \frac{2|G|e^{-\frac{m}{4} g_{\alpha,\beta}^*}}{(1-\tilde x^2)\sqrt{\det H}}\,.
$$
The above two displays yield the desired result.
\end{proof}

\subsection{Covariance}

The covariance matrix $\Sigma=\E_{\alpha, \beta}[\sigma \sigma^\top]$ captures the block structure of IBM and thus plays a major role in the statistical applications of Section~\ref{SEC:stat}. 
Moreover, the coefficients of $\Sigma$ can be expressed explicitely in terms of the local magnetization $\mu_S$ and $\mu_{\bar S}$. 

\begin{lemma}
\label{LEM:sigmaexact}
Let $\Sigma=\E_{\alpha, \beta}[\sigma \sigma^\top]$ denote the covariance matrix of a random configuration $\sigma \sim \p_{\alpha,\beta}$. For any $i \neq j \in [p]$, it holds
\begin{align*}
\Delta &:= \Sigma_{ij} = \frac{m}{2(m-1)}\E[ \mu_S^2 + \mu_{\bar S}^2] - \frac{1}{m-1}&&\text{if}   \ \ i\sim j\\
\Omega &:= \Sigma_{ij} = \E[\mu_S \mu_{\bar S}]  &&\text{if} \ \ i\nsim j\,.
\end{align*}
\end{lemma}
\begin{proof}
In this proof, we rely on symmetry of the problem: all the spins $\sigma_i$ in a  given block, $S$ or $\bar S$ have the same marginal distribution. Fix $i \neq j$.

If $i \sim j$, for example if $i,j \in S$, we have by linearity of expectation.
\begin{align*}
\Sigma_{ij}=\E[\sigma_i \sigma_j]&=\frac{1}{m(m-1)}\big(\E\sum_{(i,j) \in S^2}\sigma_i  \sigma_j-m\big)=\frac{m}{m-1} \E[\mu_S^2 ]- \frac{1}{m-1}\,.
\end{align*}
Since $\mu_S$ and $\mu_{\bar S}$ are identically distributed, we obtain the desired result.

For any $i \nsim j$ we have
\[
\Sigma_{ij}=\E[\sigma_i \sigma_j]=\frac{1}{m^2} \E\sum_{(i,j) S\times\bar S} \sigma_i  \sigma_j=\E[\mu_S\mu_{\bar S}]\,, .
\]
\end{proof}

Unlike many models in the statistical literature, computing $\Sigma$ exactly is difficult in the IBM. In particular, it is not immediately clear from \Cref{LEM:sigmaexact} that $\Delta > \Omega$, while this should be intuitively true since $\beta>\alpha$ and therefore the spin interactions are stronger within blocks than across blocks. It turns out that this simple fact can be checked by other means (see \Cref{LEM:klsingle}) for any $m\ge 2$. In the rest of this subsection, we use asymptotic approximations as $m \to \infty$ to prove  effective upper and lower bound on the gap $\Delta -\Omega$.

\begin{proposition}
  \label{prop:gap}
Let $\Delta$ and $\Omega$ be defined as in \Cref{LEM:sigmaexact} and recall that $G$ denotes the set of ground states of the IBM. Then
$$
\Delta-\Omega\simeq_m\frac{1}{2|G|}\sum_{(\tilde x, \tilde y) \in G}(\tilde x -\tilde y)^2 + \frac{1}{m}\big(\frac{(\beta-\alpha)(1-\tilde x^2)^2}{2-(\beta-\alpha)(1-\tilde x^2)}\big)\,.
$$
In particular, 
\begin{itemize}[leftmargin=*]
\item  If $\beta+|\alpha| < 2$, then $\DS \Delta-\Omega\simeq_m\frac{1}{m}\big(\frac{\beta-\alpha}{2-(\beta-\alpha)}\big)$.
\item If $\beta +|\alpha|> 2$, then  three cases arise:
\begin{enumerate}
\item if $\alpha=0$, then $\DS \Delta-\Omega\simeq_m\tilde x^2$,
\item if $\alpha>0$, then $\DS \Delta-\Omega\simeq_m\frac{1}{m}\big(\frac{(\beta-\alpha)(1-\tilde x^2)^2}{2-(\beta-\alpha)(1-\tilde x^2)}\big)>0$
\item if $\alpha<0$, then $\DS \Delta-\Omega\simeq_m2\tilde x^2$\,.
\end{enumerate}
\end{itemize}
\end{proposition}

\begin{proof}
It follows from  \cref{LEM:sigmaexact} that 
$$
\Delta=\frac{m}{m-1} \E_{\alpha, \beta}[\varphi(\mu)]-\frac{1}{m-1}
$$	
where $\varphi(\mu)=\|\mu\|_2^2/2$. Therefore, using \Cref{lem:unwprob}, we get that for $Z \sim \cN_2(0,I_2)$,
\begin{align*}
\Delta&\simeq_m\Big( 1+\frac{1}{m}\Big)\frac{1}{2|G|}\sum_{\tilde s \in G}\|\tilde s\|^2_2 + \frac{2}{m}\E\|H^{-1/2}Z\|^2_2 -\frac{1}{m}\\
&=\Big( 1+\frac{1}{m}\Big)\frac{1}{2|G|}\sum_{(\tilde x, \tilde y) \in G}(\tilde x^2+\tilde y^2) + \frac{2}{m}\Tr(H^{-1}) -\frac{1}{m}\,.
\end{align*}
Using the same argument, we get that
$$
\Omega\simeq_m\frac{1}{|G|}\sum_{(\tilde x, \tilde y) \in G} \tilde x\tilde y +\frac{4}{m}e_1^\top H^{-1}e_2 \,,
$$
where $e_1=(1,0)^\top$ and $e_2=(0,1)^\top$ are the vectors of the canonical basis of $\R^2$. Therefore
$$
\Delta-\Omega\simeq_m\frac{1}{2|G|}\sum_{\tilde s \in G}(\tilde x -\tilde y)^2+ \frac{2}{m}v^\top H^{-1} v-\frac{1}{m}(1-\tilde x^2)
$$
where $v=(1, -1)$. \cref{LEM:study_g2} implies that $v$ is an eigenvector of $H$ and thus of $H^{-1}$ and
$$
v^\top H^{-1} v=\frac{1}{\alpha-\beta +2/(1-\tilde x^2)}\,.
$$
This completes the first part of the proof and it remains only to check the different cases.
\begin{itemize}[leftmargin=*]
\item  If $\beta+|\alpha| < 2$, then $\tilde x=\tilde y=0$ is the unique ground state, which yields the result by substitution. 
\item If $\beta +|\alpha|> 2$, and
\begin{enumerate}
\item if $\alpha=0$, then $|G|=4$ and there are two ground states $(\tilde x, -\tilde x)$ and $(-\tilde x, \tilde x)$ for which $(\tilde x-\tilde y)$ does not vanish. The term in $1/m$ is negligible;
\item if $\alpha>0$, then for both ground states $(\tilde x -\tilde y)^2=0$ so that
$$
\DS \Delta-\Omega\simeq_m\frac{1}{m}\big(\frac{(\beta-\alpha)(1-\tilde x^2)^2}{2-(\beta-\alpha)(1-\tilde x^2)}\big)
$$
The fact that this quantity is positive, follows from~\eqref{EQ:approxlog} with $\gamma=0$.

\item if $\alpha<0$, then there are two ground states $(\tilde x, -\tilde x)$ and $(-\tilde x, \tilde x)$ and we can conclude as in the case $\alpha=0$ but gain a factor of $2$ because all the ground states contribute to the constant term.
\end{enumerate}
\end{itemize}
\end{proof}

It follows from \cref{prop:gap} that if $\beta + |\alpha|\neq 2$ then the covariance matrix $\Sigma$ takes two values that are separated by a term of order at least $1/m$ and even sometimes of order 1. In the next section, we leverage this information to derive statistical results.

\section{Clustering in the Ising blockmodel}
\label{SEC:stat}

In this section, we focus on the following clustering task: given $n$ i.i.d observations drawn from $\p_{\alpha, \beta}$, recover the partition $(S,\bar S)$. To that end, we build upon the probabilisitic analysis of the IBM that was carried out in the previous section in order to study the properties of an efficient clustering algorithm together with the fundamental limitations associated to this task.

\subsection{Maximum likelihood estimation}
\label{sub:mle}

Fix a sample size $n \ge 1$. Given $n$ independent copies
$\sigma^{(1)},\ldots,\sigma^{(n)}$ of $\sigma \sim\p_{\alpha,\beta}$, the log-likelihood is given by
$$
\mathcal{L}_n(S) = \sum_{t=1}^n\log\big( \p_{\alpha,\beta}(\sigma^{(t)})\big) = -n \log{Z_{\alpha,\beta}} - \sum_{t=1}^m \HIBM(\sigma^{(t)})\, .
$$
where $Z_{\alpha,\beta}$ is the partition function defined in~\eqref{EQ:partition} and $\HIBM$ is the IBM Hamiltonian defined in~\eqref{EQ:defH}. While both $Z_{\alpha, \beta}$ and $\HIBM$ could depend on the choice of the block $S$, it turns out that $Z_{\alpha, \beta}$ is constant over choices of $S$ such that $|S|=m=p/2$.

\begin{lemma}
\label{LEM:partfunc}
The partition function $Z_{\alpha, \beta}=Z_{\alpha, \beta}(S)$ defined in~\eqref{EQ:partition} is such that $Z_{\alpha, \beta}(S)=Z_{\alpha, \beta}([m])$ for all $S$ of size $|S|=m$. This statement remains true even if $m \neq p/2$.
\end{lemma}
\begin{proof}
Fix $S \subset [p]$ such that $|S|=m$ and denote by $\pi: [p] \to [p]$  any bijection that maps $[m]$ to $S$. By~\eqref{EQ:partition} and~\eqref{EQ:defH2}, it holds
\begin{align*}
&Z_{\alpha, \beta}(S)=\sum_{\sigma \in \dhyp}\exp\Big[\frac{1}{4m}\Big(2\alpha(\sigma^\top\bone_S)(\sigma^\top\bone_{\bar S}) - \beta  \big((\sigma^\top\bone_S)^2 +(\sigma^\top\bone_{\bar S})^2  \big)\Big)\Big]\\
&=\sum_{\substack{\tau=\pi(\sigma)\\\sigma \in \dhyp}}\exp\Big[\frac{1}{4m}\Big(2\alpha(\tau^\top\bone_S)(\tau^\top\bone_{\bar S}) - \beta  \big((\tau^\top\bone_S)^2 +(\tau^\top\bone_{\bar S})^2  \big)\Big)\Big]
\end{align*}
since $\pi$ is a bijection.  Moreover, $\tau^\top \bone_S=\pi(\sigma)^\top \bone_S=\sigma^\top \bone_{[m]}$ and $\tau^\top \bone_{\bar S}=\sigma^\top \bone_{\overline{[m]}}$. Hence $Z_{\alpha, \beta}(S)=Z_{\alpha, \beta}([m])$\,.

\end{proof}

Because of the above lemma, we simply write $Z_{\alpha, \beta}=Z_{\alpha, \beta}(S)$ to emphasize the fact that the  partition function does not depend on $S$. It turns out that the log-likelihood is a simple function of $S$. Indeed, define the matrix  $Q=Q_S\in \R^{p \times p}$ such that $Q_{ij} = \frac{\beta}{p}$ for $i\sim j$ and $Q_{ij} = \frac{\alpha}{p}$ for $i\nsim j$. Observe that~\eqref{EQ:defH2} can be written as
$$
\cH_{\alpha,\beta} (\sigma) = -\frac{1}{2}\sigma^\top Q \sigma =
-\frac{1}{2}\tr( \sigma \sigma^\top Q)\,.$$ 
This in turns implies
$$
\mathcal{L}_n(S) = -n \log{Z_{\alpha,\beta}} + \frac{n}{2} \tr[\hat \Sigma Q]\, ,
$$
where
$\hat \Sigma$ denotes the empirical covariance matrix defined in~\eqref{EQ:defECM}. Since $\alpha < \beta$, it is not hard to see that the likelihood maximization problem $\max_{S\subset [p], |S|=m}\mathcal{L}_n(S)$ is equivalent to
\begin{equation}
\label{EQ:MLE}
\max_{V \in \cP}\tr[ \hat \Sigma V]\,,
\qquad 
\cP  = \{vv^\top : v \in \dhyp ,  v^\top\bone_{[p]} = 0\} \,.
\end{equation}
In particular, estimating the blocks $(S,\bar S)$ amounts to estimating $v_S v_S^\top  \in \cP$, where $v_S=\bone_S-\bone_{\bar S} \in \dhyp$. Note that $v_S v_S^\top = v_{\bar S} v_{\bar S}^\top$.For an adjacency matrix $A$, the optimization problem $\max_{V \in \cP}\tr[ A V]$ is a special case of the \emph{Minimum Bisection} problem and it is known to be NP-hard in general~\citep{GarJohSto76}. To overcome this limitation, various approximation algorithms were suggested over the years, culminating with a poly-logarithmic approximation algorithm~\citep{FeiKra02}. Unfortunately, such approximations are not directly useful in the context of maximum likelihood estimation. Nevertheless, the maximum likelihood estimation problem at hand is not worst case, but rather a random problem. It can be viewed as a variant of the planted partition model (aka stochastic blockmodel) introduced in~\citep{DyeFri89}. Indeed the block structure of  $\Sigma$ unveiled in \Cref{LEM:sigmaexact} can be viewed as similar to the adjacency matrix of a weighted graph with a small bisection. Moreover, $\hat \Sigma$ can be viewed as the matrix $\Sigma$ \emph{planted} in some noise. Here, unlike the original planted partition problem, the noise is correlated and therefore requires a different analysis. In random matrix terminology, the observed matrix in the stochastic block model is of Wigner type, whereas in the IBM, it is of Wishart type. It is therefore not surprising that we can use the same methodology in both cases. In particular, we will use the semidefinite relaxation to the {\sf MAXCUT} problem of ~\cite{GoeWil95} that was already employed in the planted partition model~\citep{AbbBanHal16, HajWuXu16}.

It can actually be impractical to use directly the matrix $\hat \Sigma$ in the above relaxations, and we apply a pre-preprocessing that  amounts to a centering procedure, which simplifies our analysis. Given $\sigma \in \dhyp$, define its centered version $\bar \sigma$ by
\[
\bar \sigma= \sigma - \frac{\mathbf{1}_{[p]}^\top \sigma}{p} \mathbf{1}_{[p]} = P \sigma\,,
\]
where $P=I_p -\frac1p \bonep \bonep^\top$ is the projector onto the subspace orthogonal to $\bone_{[p]}$.
Moreover, let $\Gamma=P\Sigma P$ and  $\hat \Gamma=P\hat \Sigma P$  respectively denote the covariance and empirical covariance matrices of the vector $\bar \sigma$.

Note that for all $V\in \cP$, we have that  $\tr[ \hat \Gamma  V]=\tr[ \hat \Sigma  V] $ since $V \bone_{[p]} \bone^\top_{[p]} = 0$, so that $PVP=V$. It implies that the likelihood function is unchanged over $\cP$ when substituting $\hat \Sigma$ by $\hat \Gamma$. %
Moreover, $\E[\hat \Gamma]=\Gamma$ and the spectral decomposition of $\Gamma$ is given by
\begin{equation}
\label{EQ:Gamma}
\Gamma = (1-\Delta) P + p\frac{\Delta-\Omega}{2}u_S u_S^\top\,,
\end{equation}

where $u_S = v_S/\sqrt{p}$ is a unit vector. Therefore the matrix $\Gamma$ has leading eigenvalue  $(1-\Delta) + p (\Delta-\Omega)/2$ with associated unit eigenvector $u_S$. Moreover,  its eigengap is $p (\Delta-\Omega)/2$. It is well known in matrix perturbation theory that the eigengap plays a key role in the stability of the spectral decomposition of $\Gamma$ when observed with noise.

\subsection{Exact recovery via semidefinite programming}

In this subsection, we consider the following semi-definite programming (SDP) relaxation of the optimization problem~\eqref{EQ:MLE}:
\begin{equation}
\label{EQ:SDP}
\max_{V \in \cE}\tr[ \hat \Gamma V]\,,
\qquad 
  \mathcal{E} = \big\{V \in \cS_p : \diag (V) =\mathbf{1}_{[p]}, V \succeq 0\big\} \,,
\end{equation}
where $\cS_p$ denotes the set of $p \times p$ symmetric real matrices.
The set $\cE$ is the set of correlation matrices, and it is known as the \emph{elliptope}. We recall the definition of the vector $v_S=\bone_S-\bone_{\bar S} \in \dhyp$  and note that $v_S v_S^\top \in \mathcal{P}  \subset \mathcal{E}$.
Moreover, we denote by $\hat V^{\textsf{SDP}}$ any solution to the the above program. Our goal is to show that~\eqref{EQ:SDP} has a unique solution given by
$\hat V^{\textsf{SDP}} = v_Sv_S^\top$, i.e., the SDP relaxation is tight. In contrast to the MLE, this estimator can be computed efficiently by interior-point methods \citep{BoyVan04}.

While the dual certificate approach of \cite{AbbBanHal16} could be used in this case (see also~\cite{HajWuXu16}) we employ a slightly different proof technique, more geometric, that we find to be more transparent. This approach is motivated by the idea that the relaxation is tight in the population case, suggesting that it might be the case as well when $\hat \Gamma$ is close to $\Gamma$.

Recall that for any $X_0 \in \mathcal{E}$, the normal cone to
$\mathcal{E}$ at $X_0$ is denoted by $\mathcal{N}_{\mathcal{E}}(X_0)$ and defined by
$$
\mathcal{N}_{\mathcal{E}}(X_0) = \big\{ C \in \cS_p\, : \, \tr(CX) \le \tr(CX_0) \, , \, \forall X \in \mathcal{E}\big\} \, .
$$
It is the cone of matrices $C \in \cS_p$  such that $\max_{X \in \mathcal{E}} \tr(CX)=\tr(CX_0)$. Therefore,  $v_Sv_S^\top$ is a solution of~\eqref{EQ:SDP}, i.e., the SDP relaxation is tight,  whenever
$\hat \Gamma \in \mathcal{N}_{\mathcal{E}}(v_S v_S^\top)$. The normal cone can be described using the following Laplacian operator. For any matrix $C \in \cS_p$, define
$$
L_S(C) :=   \mathbf{diag}(Cv_Sv_S^\top) - C,
$$
and observe that $L_S(C)v_S=0$. Indeed, since $v_S \in \dhyp$, it holds,
$$
\mathbf{diag}(Cv_Sv_S^\top)v_S=\mathbf{diag}(Cv_S\bone_{[p]}^\top)\bone_{[p]}=Cv_S\,.
$$

\begin{proposition}
\label{PRO:conelaplace}
For any matrix $C \in \cS_p$, the following are equivalent
\begin{enumerate}
\item $C \in \mathcal{N}_{\mathcal{E}_p}(v_S v_S^\top)\, .$
\item $L_{S}(C)=\mathbf{diag}(Cv_S v_S^\top) - C\succeq 0 \,,$
\end{enumerate}
Moreover, if $L_{S}(C)\succeq 0$ has only one eigenvalue equal to 0, then $v_S v_S^\top$ is the unique maximizer of $\tr(CV)$ over $V \in \mathcal{E}$.
\end{proposition}
\begin{figure}
  \centering \includegraphics[width=0.4\textwidth]{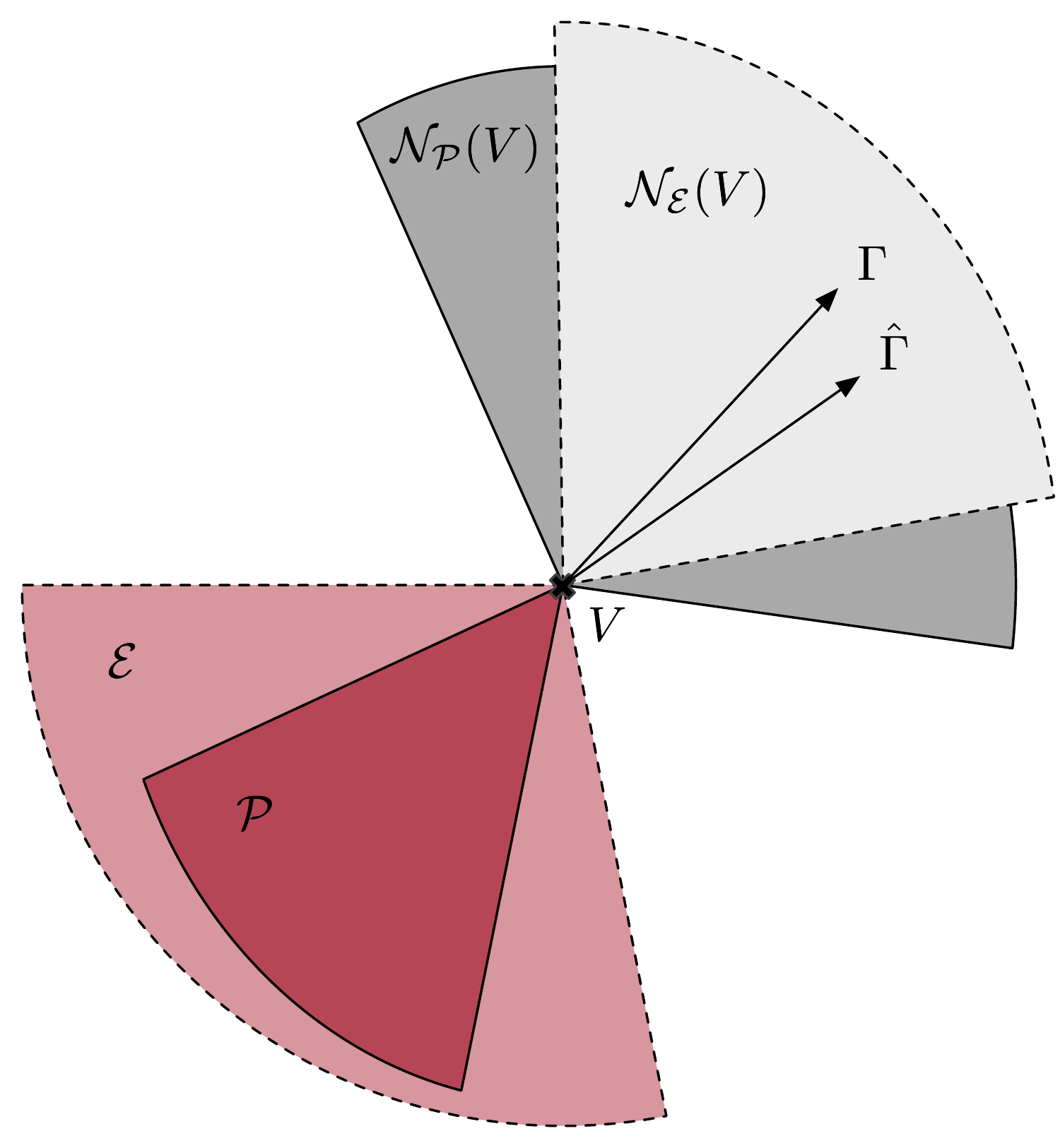}
\caption{The geometric interpretation for the analysis of this convex relaxation. In the population case, the true value of the parameter $V=v_Sv_S^\top$ is the unique solution of both the maximum likelihood problem on $\mathcal{P}$ and of the convex relaxation on $\mathcal{E}$, as $\Gamma$ belongs to both normal cones at $V$. The relaxation is therefore tight with $\Gamma$ as input. We show that when the sample size is large enough, the sample matrix $\hat \Gamma$ is close enough to $\Gamma$ and also in both normal cones, making $V$ the solution to both problems.}
\label{FIG:cones}
\end{figure}
\begin{proof}
It is known~\citep[see][]{LauPol96} that the normal cone $\mathcal{N}_{\mathcal{E}}(v_S v_S^\top)$ is given by
$$
\mathcal{N}_{\mathcal{E}}(v_S v_S^\top) = \big\{ C \in \cS_p \, : \, C = D-M , D \; \text{diagonal, }, M \succeq 0 , v_S^\top Mv_S  = 0\big\}\, ,
$$
where $M\succeq 0$ denotes that $M$ is a symmetric, semidefinite positive matrix. We are going to make use of the following facts. First for any diagonal matrix $D$ and any $V \in \cE$, it holds $\diag (DV) = D$. Second, taking $V=v_Sv_S^\top$, we get 
$$L_S(C) v_Sv_S^\top = \diag (Cv_Sv_S^\top)v_Sv_S^\top - Cv_Sv_S^\top=\diag (Cv_Sv_S^\top)-Cv_Sv_S^\top\,,$$ so that
\begin{equation}
\label{EQ:prcone1}
\diag (L_S(C)v_Sv_S^\top)=0\,.
\end{equation}

$2. \Rightarrow 1.$ Let $C \in v_S^\top$ be such that $L_S(C)\succeq 0$. By definition, we have $C=\diag (Cv_S v_S^\top)-L_S(C)$ and it remains to check that $ v_S^\top L_S(C)v_S = 0$, which follows readily from~\eqref{EQ:prcone1} with $V=v_Sv_S^\top$.

$1. \Rightarrow 2.$ Let
$C=D-M \in \mathcal{N}_{\mathcal{E}_p}(v_S v_S^\top)$ where $D$ is diagonal and $M \succeq 0$,  $v_S^\top M v_S = 0$, which implies that $Mv_S=0$. It yields, $Cv_S v_S^\top = Dv_S v_S^\top$ and
$\diag (Cv_S v_S^\top) = \diag (Dv_S v_S^\top) = D$ so that the
decomposition is  necessarily $D= \diag (Cv_S v_S^\top)$
and $M = L_S(C) = \diag (Cv_S v_S^\top) - C$. In particular, $L_S(C) \succeq 0$.

\medskip

Thus, if $L_S(C)\succeq 0$ then  $v_S v_S^\top$ is a maximizer of $\tr(CV)$ over $V \in \mathcal{E}$. To prove uniqueness, recall that for any  maximizer $V \in \mathcal{E}$, we have $\tr(CV)=\tr(Cv_S v_S^\top)$. Plugging $C = \diag (Cv_S v_S^\top) - L_S(C)$ and using~\eqref{EQ:prcone1} yields
\begin{align*}
\tr( \diag (Cv_S v_S^\top) V)- \tr(L_S(C)V)&=\tr( \diag (Cv_S v_S^\top) v_S v_S^\top)\\
&=\tr( \diag (Cv_S v_S^\top))\,.
\end{align*}
Recall that $\tr(\diag (Cv_S v_S^\top) V)=\tr(\diag (Cv_S v_S^\top))$ so that the above display yields $\tr(L_S(C)V)=0$. Since $V\succeq 0$ and the kernel of the semidefinite positive matrix $L_S(C)$ is spanned by $v_S$, we have that $V=v_S v_S^\top$.

\end{proof}

It follows from \Cref{PRO:conelaplace} that if $L_S(\hat \Gamma) \succeq 0$ and has only one eigenvalue equal to zero, then $v_Sv_S^\top$ is the solution to~\eqref{EQ:SDP}. In particular, in this case, the SDP allows exact recovery of the block structure $(S,\bar S)$. Observe that the conditions of \Cref{PRO:conelaplace} hold if $\hat \Gamma$ is replaced by the population matrix $\Gamma$. Indeed, using~\eqref{EQ:Gamma}, we obtain
\begin{align*}
L_S(\Gamma)&=\big(1-\Delta + p\frac{\Delta-\Omega}{2}\big)I_p - (1-\Delta)P -p \frac{\Delta-\Omega}{2}u_Su_S^\top\\
&=(1-\Delta) \frac{\bone_{[p]}}{\sqrt{p}} \frac{\bone^\top_{[p]}}{\sqrt{p}}  -p \frac{\Delta-\Omega}{2}u_Su_S^\top + p\frac{\Delta-\Omega}{2}I_p\,,
\end{align*}
where we used the fact that $I_p-P$ is the projector onto the linear span of $\bone_{[p]}$.
Therefore, the eigenvalues of $L_S(\Gamma)$ are $0$, $1-\Delta+ p(\Delta-\Omega)/2 $, both with multiplicity 1 and $p(\Delta-\Omega)/2$ with multiplicity $p-1$. In particular, for $p \ge 2$, $L_S(\Gamma)\succeq 0$ and it has only one eigenvalue equal to zero.

Extending this result to $L_S(\hat \Gamma)$ yields the following theorem, as illustrated in Figure~\ref{FIG:cones}. Let   $C_{\alpha,\beta}>0$ be a positive constant such that $\Delta-\Omega>C_{\alpha , \beta} /p$. Note that such a constant $C_{\alpha, \beta}$ is guaranteed to exist in view of \Cref{prop:gap}.

\begin{theorem}
\label{thm:sdp}
The SDP relaxation~\eqref{EQ:SDP} has a unique maximum at $V=v_Sv_S^\top$ with probability $1-\delta$ whenever
\[
n >   16\big(3+\frac{2}{C_{\alpha,\beta}}\big)\frac{\log(4p/\delta)}{\Delta-\Omega}(1+o_p(1))\, .
\]
In particular, the SDP relaxation recovers exactly the block structure $(S,\bar S)$. 
\end{theorem}
\begin{proof}
Recall that $L_S(\hat \Gamma)v_S=0$ and any $C \in \cS_p$, denote by  $\lambda_2[C]$ its second smallest eigenvalue. Our goal is to show that $\lambda_2[L_S(\hat \Gamma)]>0$. To that end, observe that
$$
L_S(\hat \Gamma)=L_S(\Gamma) +\diag\big((\hat \Gamma - \Gamma)v_Sv_S^\top\big) + \Gamma -\hat \Gamma\,. 
$$
Therefore, using from Weyl's inequality and the fact $\lambda_2[L_S(\Gamma)]=p(\Delta-\Omega)/2$, we get
\begin{equation}
\label{EQ:lambda2}
\lambda_2[L_S(\hat \Gamma)] \ge  p\frac{\Delta-\Omega}{2}-\|\diag\big((\hat \Gamma - \Gamma)v_Sv_S^\top\big)\|_{\mathrm{op}}-\|\hat \Gamma-\Gamma \|_{\mathrm{op}}\,,
\end{equation}
where $\|\cdot\|_{\mathrm{op}}$ denotes the operator norm. Therefore, it is sufficient to upper bound the above operator norms.
This is ensured by the following Lemma.

\begin{lemma}
\label{lem:gammanorms}
Fix $\delta >0$ and define 
$$
\mathcal{R}_{n,p}(\delta)=2p\max\Big(\sqrt{\frac{(1+2/C_{\alpha,\beta})(\Delta-\Omega) \log(4p/\delta)}{n}} \, , \, \frac{(6+4/C_{\alpha, \beta}) \log(p/\delta)}{n} \Big)\,.
$$

With probability $1-\delta$, it holds simultaneously that
\begin{equation}
\label{EQ:normop}
\|\hat \Gamma-\Gamma \|_{\mathrm{op}} \le \mathcal{R}_{n,p}(\delta)(1+o_p(1))\, .
\end{equation}
and
\begin{equation}
\label{EQ:dormdiag}
\|\diag\big((\hat \Gamma - \Gamma)v_Sv_S^\top\big)\|_{\mathrm{op}} \le \mathcal{R}_{n,p}(\delta)(1+o_p(1))\,. 
\end{equation}
\end{lemma}
\begin{proof}

To prove~\eqref{EQ:normop}, we use a Matrix Bernstein inequality for sum of independent matrices from \cite{Tro15}. To that end, note that
$$
\hat \Gamma - \Gamma=\frac1n\sum_{t=1}^n M_t\,,
$$
where $M_1, \ldots, M_n$ are i.i.d random matrices given by $M_t = (\bar \sigma^{(t)}\bar \sigma^{(t) \, \top} - \Gamma)$, $t=1, \ldots, n$.
We have
\[
\|M_t\|_{\mathrm{op}} \le \|\bar \sigma^{(t)}\bar \sigma^{(t) \, \top}\|_{\mathrm{op}} + \|\Gamma\|_\mathrm{op} \le p+\|\Gamma\|_\mathrm{op}\, .
\]
Furthermore, we have that
\begin{align*}
\E[M_t^2] &=\E[\|\bar \sigma^{(t)}\|^2 \bar \sigma^{(t)}\bar \sigma^{(t) \, \top} - \bar \sigma^{(t)}\bar \sigma^{(t) \, \top} \Gamma - \Gamma \bar\sigma^{(t)}\bar \sigma^{(t) \, \top} + \Gamma^2]\\
&= p \E[\bar \sigma^{(t)}\bar \sigma^{(t) \, \top}]  - \Gamma^2 -\Gamma^2 +\Gamma^2\preceq p \Gamma\, .
\end{align*}
As a consequence, $\sum_{t=1}^n \E[M_t^2] \preceq p\Gamma$. By Theorem 1.6.2 in \cite{Tro15}, this yields
\begin{equation}
\label{EQ:tropp1}
\p\big(\|\hat \Gamma - \Gamma\|_{\mathrm{op}} >t \big)\le 2p \exp\Big(-\frac{nt^2}{2p \|\Gamma\|_{\mathrm{op}} + 2(p+\|\Gamma\|_{\mathrm{op}})t} \Big)\, .
\end{equation}
We have $\|\hat \Gamma - \Gamma\|_{\mathrm{op}} \le t$ with probability $1-\delta$ for any $t$ such that
\[
\log(2p/\delta) \le \frac{nt^2}{2p \|\Gamma\|_{\mathrm{op}} + 2(p+\|\Gamma\|_{\mathrm{op}})t} \, .
\]
This holds for all
\[
t \le \max\Big(\sqrt{\frac{4p \|\Gamma\|_{\mathrm{op}} \log(2p/\delta)}{n}} \, , \, \frac{4(p + \|\Gamma\|_{\mathrm{op}}) \log(2p/\delta)}{n} \Big)\, .
\]
To conclude the proof of~\eqref{EQ:normop}, observe that
$$
\|\Gamma\|_{\mathrm{op}}=p\frac{\Delta-\Omega}{2} + 1-\Delta \le \big(1+\frac1{C_{\alpha, \beta}}\big)(\Delta-\Omega)p\,,
$$
where $C_{\alpha, \beta}>0$ is defined immediately before the statement of \Cref{thm:sdp}.

\medskip

We now turn to the proof of~\eqref{EQ:dormdiag}. Recall that $v_S \in \dhyp$ so that the $i$th diagonal element is given by
$$
\diag\big((\hat \Gamma - \Gamma)v_Sv_S^\top\big)_{ii}=e_i^\top (\hat \Gamma - \Gamma)v_S\,,$$
where $e_i$ denotes the $i$th vector of the canonical basis of $\R^p$. Hence,
$$
\|\diag\big((\hat \Gamma - \Gamma)v_Sv_S^\top\big)\|_{\mathrm{op}}=\max_{i \in [p]}\big|\diag\big((\hat \Gamma - \Gamma)v_Sv_S^\top\big)_{ii}\big|=\max_{i \in [p]}|e_i^\top (\hat \Gamma - \Gamma)v_S| \,.
$$
We bound the right hand-side of the above inequality by noting that
\[
e_i^\top (\hat \Gamma - \Gamma)v_S = \frac{m}{n}\sum_{t=1}^n \big( \bar \sigma_i^{(t)} (\mu_S^{(t)}-\mu_{\bar S}^{(t)}) - \E[\bar \sigma_i^{(t)} (\mu_S^{(t)}-\mu_{\bar S}^{(t)})] \big) \,,
\]
where $\mu_S^{(t)}=\bone_S^\top\bar \sigma^{(t)}/m \in [-1,1]$ and $\mu_{\bar S}^{(t)}$ is defined analogously. 
The random variables $\bar \sigma_i^{(t)} (\mu_S^{(t)}-\mu_{\bar S}^{(t)}) - \E[\bar \sigma_i^{(t)} (\mu_S^{(t)}-\mu_{\bar S}^{(t)})]$ are centered, i.i.d., and are bounded in absolute value by 2 for all $t \in [n]$. Moreover, it follows from \Cref{LEM:sigmaexact} that the variance of these random variables is bounded by 
$$
\E[(\mu_S^{(t)} - \mu_{\bar S}^{(t)})^2] \le 2(\Delta-\Omega)+\frac4p=:\nu^2\,.
$$
By a one-dimensional Bernstein inequality, and a union bound over $p$ terms, we have therefore that
\[
\p\big(\max_{i \in [p]} |e_i^\top (\hat \Gamma - \Gamma)v_S| >\frac{p t}{n} \big) \le 2p \exp\Big(-\frac{ t^2/2}{n \nu^2 + 2t/3}\Big)\, .
\]
which yields
\[
\max_{i \in [p]} |e_i^\top (\hat \Gamma - \Gamma)v_S| \le p \max\Big(\sqrt{\frac{2\nu^2 \log(2p/\delta)}{n}}, \frac{4 \log(2p/\delta)}{3n} \Big)\,,
\]
with probability $1-\delta$. It completes the proof of~\eqref{EQ:dormdiag}.
\end{proof}

To conclude the proof of \Cref{thm:sdp}, note that for the prescribed choice of $n$, we have
$$
2\mathcal{R}_{n,p}(\delta)(1+o_p(1))<p\frac{\Delta-\Omega}{2}
$$
and it follows from~\eqref{EQ:lambda2} that $\lambda_2[L_S(\hat \Gamma)] >0$.
\end{proof}
\begin{remark}
We have not attempted to optimize the constant term $16(3+2/C_{\alpha,\beta})$ that appears in \Cref{thm:sdp} and it is arguably suboptimal. One way to see how it can be reduced at least by a factor 2 is by noting that  the factor $p$ in the right-hand side of~\eqref{EQ:tropp1} is in fact superfluous thus resulting in a extra logarithmic factor in~\eqref{EQ:normop}. This is because, akin to the stochastic blockmodel analysis in~\cite{AbbBanHal16}, the matrix deviation inequality from~\cite{Tro15} is too coarse for this problem. The extra factor $p$ may be removed using the   concentration results of \Cref{sub:concentration} but at the cost of a much longer argument. Indeed, using~\Cref{lem:unwprob}, we can establish the concentration of local magnetization around the ground states and conditionally on these magnetizations, the configurations are uniformly distributed. These conditional distributions can be shown to exhibit sub-Gaussian concentration so that $\sigma^\top u$ and thus $\bar \sigma^\top u$ are sub-Gaussian with constant variance proxy for any unit vector $u \in \R^p$. This result can yield a bound for $\|\hat \Gamma-\Gamma \|_{\mathrm{op}}$ using an $\eps$-net argument that is standard in covariance matrix estimation. With this in mind, we could get an upper bound in~\eqref{EQ:normop} that is negligible with respect to $\mathcal{R}_{n,p}$ thereby removing a factor 2. Nevertheless, in absence of a tight control of the constant $C_{\alpha, \beta}$, exact constants are hopeless and beyond the scope of this paper.
\end{remark}

Combined with \Cref{prop:gap} that quantifies the gap $\Delta-\Omega$ in terms of the dimension $p$, \Cref{thm:sdp} readily yields the following corollary.
\begin{corollary}
\label{cor:sdp}
There exists positive constants $C_1$ and $C_2$ that depend on $\alpha$ and $\beta$ such that the following holds. The SDP relaxation~\eqref{EQ:SDP} recovers the block structure $(S, \bar S)$ exactly with probability $1-\delta$ whenever 
\begin{enumerate}
\item $\DS n \ge C_1 p\log(p/\delta)$ if $\beta + |\alpha| <2$ or $\alpha >0$
\item $\DS n \ge C_2 \log(p/\delta)$ otherwise. 
\end{enumerate}
In particular, if $\beta-\alpha >2, \alpha \le 0$ a number of observations that is logarithmic in the dimension $p$ is sufficient to recover the blocks exactly. 
\end{corollary}

These results suggest that there is a sharp phase transition in sample complexity for this problem, depending on the value of the parameters $\alpha$ and $\beta$. We address this question further in Section~\ref{SEC:conclusion}. The last subsection shows that these rates are, in fact, optimal.
\subsection{Information theoretic limitations}
\label{sub:lb}

In this section, we present lower bounds on the sample size needed to recover the partition $(S,\bar S)$ and compare them to the upper bounds of \Cref{thm:sdp}. In the sequel, we write $\hat S \asymp S$ if either $(\hat S, \bar{\hat S})= (S, \bar{S})$ or $(\hat S, \bar{\hat S})= (\bar{S}, S)$ to indicate that the two partitions are the same. We write $\hat S \not\asymp S$ to indicate that the two partitions are different. 

For any balanced partition $(S,\bar S)$, consider a ``neighborhood'' $\cT_S$ of $(S, \bar S)$ composed of balanced partitions such that for all $(T,\bar T) \in \cT_S$, we have $\rho(S,T) = 1$ and $\rho(\bar S, \bar T)=1$. We first compute the Kullback--Leibler divergence between the distributions $\p_S$ and $\p_T$.

\begin{lemma}
\label{LEM:klsingle}
For any positive $\beta$, $\alpha<\beta$, and $T \in \cT_S$, it holds that
\[
\KL(\p_T,\p_S) = \frac{p-2}{p}(\beta-\alpha)(\Delta-\Omega)\, .
\]

\end{lemma}
\begin{proof}
By definition of the divergence and of the distributions, we have that
\begin{align*}
\KL(\p_T,\p_S) &= \E_{T}\Big[\log\Big(\frac{\p_T}{\p_S}(\sigma)\Big)\Big]\\
&= \E_{T}\big[\tr[(Q_T-Q_S) \sigma \sigma^\top] \big]\\
&= \tr[(Q_T-Q_S) \Sigma_T] 
\end{align*}
Note that most of the coefficients of $Q_T-Q_S$ are equal to 0. In fact, noting $\{s\} = S\cap \bar T$ and $\{t\}= \bar S \cap T$, we have
\[
(Q_T - Q_S)_{ij} = \frac{\alpha -\beta}{p} \quad \text{if} \quad
\begin{cases} 
	&i \in S \setminus \{s\}\,, \, j=s\\
       &i = s\, ,\, j \in S \setminus \{s\}\\
	& i \in \bar S \setminus \{t\}\,, \, j=t\\
	&i = t\, ,\, j \in \bar S \setminus \{t\}\, 
\end{cases}
\]
and
\[
(Q_T - Q_S)_{ij} = \frac{\beta-\alpha}{p} \quad \text{if} \quad
\begin{cases} 
	&i \in S \setminus \{s\}\,, \, j=t\\
       &i = s\, ,\, j \in \bar S \setminus \{t\}\\
	& i \in \bar S \setminus \{t\}\,, \, j=s\\
	&i = t\, ,\, j \in S \setminus \{s\}\, ,
\end{cases}
\]
and $0$ otherwise. There are therefore $p-2$ coefficients of each sign. Furthermore, whenever $(Q_T - Q_S)_{ij} = (\alpha-\beta)/p$, we have $(\Sigma_T)_{ij}=\Omega$, and whenever $(Q_T - Q_S)_{ij} = (\beta-\alpha)/p$, we have $(\Sigma_T)_{ij}=\Delta$. Computing $\tr[(Q_T-Q_S) \Sigma_T]$ explicitly yields the desired result.
\end{proof}

From this lemma, we derive the following lower bound.
\begin{theorem}
\label{thm:lb}
For $\gamma \in (0,3/5)$ and $p \ge 6$ and
\[
n \le \frac{\gamma \log(p/4)}{(\beta-\alpha)(\Delta-\Omega)}.
\]
We have
\[
\inf_{\hat S} \max_{S \in \cS} \p^{\otimes n}_S\big((\hat S, \bar{\hat S})\not\asymp  (S, \bar{S})) \ge  \frac{p-2}{p}\big(1- \gamma - \sqrt{\gamma}\big)>0\, ,
\]
where the infimum is taken over all estimators of $S$. Note that the right-hand side of the above inequality goes to 1 as $p \to \infty$ and $\gamma \to 0$.
\end{theorem}

\begin{proof}
First, note that by~\Cref{LEM:klsingle}, for any $T \in \cT_S$, it holds $|\cT_S| = (p/2-1)^2$ so that
\[
\KL(\p_T^{\otimes n},\p^{\otimes n}_S) = n \KL(\p_T,\p_S) \le n (\beta-\alpha)(\Delta-\Omega) \le  \gamma \log(p/4) \le \frac{\gamma}{2} \log |\cT_S|\,.
\]
Thus  Theorem~2.5 in~\cite{Tsy09} yields
\begin{align*}
\inf_{\hat S} \max_{S \in \cP} \p^{\otimes n}_S(\hat S\not\asymp  S) &\ge \frac{\sqrt{|\cT_S|}}{1+\sqrt{|\cT_S|}} \Big(1- \gamma - \sqrt{\frac{\gamma}{\log(|\cT_S|)}} \Big)\\
&\ge \frac{p-2}{p}\big(1- \gamma - \sqrt{\gamma}\big)>0\,,
\end{align*}

for $\gamma \in (0,3/5)$.
\end{proof}

The lower bound of \Cref{thm:lb} matches the upper bounds of \Cref{thm:sdp} up to numerical constant. This indicates that the SDP relaxation studied in the paper is rate optimal: the sample complexity stated in \Cref{cor:sdp} has optimal dependence on the dimension $p$. Note that past work on exact recovery in the stochastic blockmodel~\citep{AbbBanHal16, HajWuXu16} was able to show that SDP was also optimal with respect to constants. We do not pursue this questions in the present paper.

\section{Conclusion and open problems}
\label{SEC:conclusion}

This paper introduces the Ising block model (IBM) for large binary random vectors with an underlying cluster structure. In this model, we studied the sample complexity of recovering exactly the clusters. Unsurprisingly, this paper bears similarities with the stochastic blockmodel, but also differences. For example, in the stochastic blockmodel one is given only one observation of the graph. In the IBM, given one realization $\sigma^{(1)} \in \dhyp$, the maximum likelihood estimator is the trivial clustering that assigns $i \in [p]$ to a cluster according to the sign of $\sigma^{(1)}_i$, up to a trivial reassignment to keep the partition balanced. \\

Below is a summary of our main findings:
\begin{enumerate}
\item The model exhibits three phases depending on the values taken by two parameters.
\item In one phase, where the two clusters tend to have opposite behavior, the sample complexity is logarithmic in the dimension;
 in the other two, it is near linear. These sample complexities are proved to be optimal in an information theoretic sense.
\item Akin to the stochastic blockmodel, the optimal sample complexity is achieved using the natural semidefinite relaxation to the {\sf MAXCUT} problem.
\end{enumerate}

Many questions regarding this model remain open. The first and most natural is the determination of exact constants. \Cref{thm:lb} suggests that there exists a universal constant $C^\star$ such that the optimal sample complexity is $$\frac{C^\star\log(p)}{(\beta-\alpha)(\Delta-\Omega)}(1+o_p(1))\,.$$
Throughout this paper, we have only kept loosely track of the correct dependency of the constants as function of the constants $(\alpha, \beta)$. We have shown that the optimal sample complexity is a product of $\log(p)/(\Delta-\Omega)$ and of a constant term that only becomes arbitrarily large when $\alpha$ is very close to $\beta$, with a divergence of order $(\beta-\alpha)^{-1}$, which is consistent with our lower bound. In the spirit of exact thresholds for the stochastic blockmodel~\citep{Mas14,  MosNeeSly15,AbbBanHal16}, we find that proving existence of the constant $C^\star$ and computing it worthy of investigation but is beyond the scope of the present paper.

Another possible development is the extension of this model to settings with multiple blocks, possibly of unbalanced sizes. This has been studied in the case of the stochastic blockmodel for graphs in the sparse case \citep{AbbSan15,BanMoo16} and in the dense case \citep{RohChaYu11,GaoMaZha15,GaoMaZha16}. For the  Ising blockmodel, the main challenge is that the population covariance matrix cannot be directly computed from the parameters of the problem, and an analysis of the ground states of the free energy is required. Developing a general approach to this task, rather than having to do an ad hoc analysis for each case would be an important step in this direction.

We have only analyzed in this work the performance of the semidefinite positive relaxation of the maximum likelihood problem, but other methods can be considered for total or partial recovery. In related problems, belief propagation is used to recover communities \citep[see e.g.][and work cited above]{MosNeeSly14,MoiPerWei15,AbbSan16,AbbSan16b,LesKrzZde17}. In particular, \cite{LesKrzZde17} covers Hopfield models, which are a generalization of our model. Another possible venue is the use of greedy random algorithms, which have been used to find local solutions of {\sf MAXCUT} in \cite{AngBubPer16}. It is possible that studying these types of algorithms is necessary in order to obtain sharper rates.

\medskip

Finally, in view of the simple spectral decomposition~\eqref{EQ:Gamma} of $\Gamma$, one may wonder about the behavior of the a simple method that consists in computing the leading eigenvector of $\hat \Gamma$ and clustering according to the sign of its entries. Such a method is the basis of the approach in denser graph models in \cite{McS01} or \cite{AloKriSud98}. The results of such an approach are easily implementable as follows. 

Let $\hat u$ denote a leading unit eigenvectors of $\hat \Gamma$ and consider the following estimate for the partition $(S, \bar S)$:
\begin{equation}
\label{EQ:spectral}
\hat S \asymp  \{ i \in [p] \; | \; \hat u_i >0 \}\, .
\end{equation}
It follows from the Perron-Frobenius theorem   that  $\hat S \asymp S$ whenever $\mathrm{sign}(\hat \Gamma)=\mathrm{sign}(\Gamma)$.
This allows for perfect recovery of $S$, but only holds with high probability when $n$ is of order $\log(p)/(\Delta-\Omega)^2$, which is suboptimal. It is however possible to obtain partial recovery guarantees for the spectral recovery. In order to state our result, for any two partitions $(S, \bar S)$, $(T, \bar T)$ define
$$
|S\Diamond T|=\min \Big( |S \bigtriangleup T | ,  |S \bigtriangleup \bar T |\Big)
$$
where $\bigtriangleup$ denotes the symmetric difference.
\begin{proposition}
\label{PRO:davkah}
Fix $\delta \in (0,1)$ and let $\hat S \subset [p]$ be defined in~\eqref{EQ:spectral}. Then, there exits a constant $\gamma_{\alpha, \beta}>0$ such that with probability $1-\delta$,
\[
\frac{1}{p}|S\Diamond \hat S|\le \gamma_{\alpha,\beta} \, \frac{\log(4p/\delta)}{n(\Delta-\Omega)}\, .
\]
\end{proposition}

\begin{proof}
Let $\hat u$ denote the leading unit eigenvector of $\hat \Gamma$ and let $\hat v=\sqrt{p}\hat u$. Recall that $v_S=\bone_S-\bone_{\bar S}$ and observe that
\begin{align*}
|S\Diamond \hat S|&=\min\Big(\sum_{i=1}^p \1(\hat v_i \cdot (v_S)_i\le 0)  ,\sum_{i=1}^p \1(\hat v_i \cdot (v_S)_i\ge 0)\Big)\\
&\le \min\big(\|\hat v - v_S\|^2, \|\hat v+ v_S\|^2\big)=p\min\big(\|\hat u - u_S\|^2, \|\hat u + u_S\|^2\big)\,,
\end{align*}
where in the inequality, we used the fact that $v_S \in \dhyp$ so that 
$$
\1(\hat v_i \cdot (v_S)_i\le 0)\le |\hat v_i - (v_S)_i|\1(\hat v_i \cdot (v_S)_i\le 0) \le |\hat v_i - (v_S)_i|^2\,.
$$
Using a variant of the Davis-Kahan lemma (see, e.g, \cite{WanBerSam16}), we get
$$
\frac{1}{p}|S\Diamond \hat S| \le \frac{ \|\hat \Gamma - \Gamma \|^2_{\mathrm{op}}}{(\lambda_1(\Gamma) - \lambda_2(\Gamma))^2}\,,
$$
and the result follows readily from~\eqref{EQ:normop} and the fact that the eigengap of $\Gamma$ is given by $p(\Delta-\Omega)/2$.
\end{proof}

In terms of exact recovery, this result is quite weak as it only gives guarantees for a sample complexity of the order of $p\log(p/\delta)/(\Delta-\Omega)$, which is suboptimal by a factor of~$p$. Moreover, for the bound of \Cref{PRO:davkah} to be non-trivial, one already needs the sample size to be of the same order as the one required for exact recovery by semi-definite programming. 
 Nevertheless \Cref{PRO:davkah} raises the question of the optimal rates of estimation of $S$ with respect to the metric $|S\Diamond \hat S|/p$. While partial recovery is beyond the scope of this paper, it would be interesting to establish the optimal rate.
 
\section*{Acknowledgements} 
 
P.R. Thanks Andrea Montanari for pointing out a connection to the Hopfield model.

\appendix

\section{Facts about the Curie-Weiss model}
\label{SUB:CW}
We begin by stating some well known facts about the Curie-Weiss model. These results are standard in the statistical physics literature and the interested reader can find more details in~\cite{FriVel16, Ell06} for example. However, the precise behavior of the free energy that we need for our subsequent analysis does not seem to be readily available in the literature so we prove below a lemma that suits our purposes.

Recall that the Curie-Weiss model is a special case of the Ising block model when $\alpha=\beta=b$. In this case, the free energy takes the form:
\begin{equation}
\label{EQ:freeNRJCW}
\gCW_b(\mu)=-2b \mu^2- 4h\big(\frac{\mu+1}{2}\big)
\end{equation}
where we recall that $\mu=\sigma^\top \bone/p$ is the global magnetization of $\sigma$. The minima $x \in (-1,1)$ of $g$ are called ground states and satisfy the first order optimality condition, also known as \emph{mean field equation}  
$$
\log\big(\frac{1+x}{1-x}\big) =2b x\,.
$$

If $b\le 1$, then the unique solution to the mean field equation
is $x=0$. Moreover, $\gCW_b$ is increasing on $[0,1]$.

 If $b>1$, then the mean field equation has two solutions
$\tilde x>0$ and $-\tilde  x$ in $(-1,1)$. In any case, these
solutions are global minima that are also the only local minima of
$\gCW_b$.  In particular, when $b > 1$, $\gCW_b$ is monotone decreasing in
the interval $(0, \tilde x)$ and monotone increasing in the interval
$(\tilde x, 1)$.

The following lemma is a refinement of these well-known facts that quantifies the curvature of $\gCW_b$ around its minima.

\begin{lemma}
\label{LEM:CW}
Fix $b>1$ in the Curie-Weiss model and denote by $\tilde x>0$ and $-\tilde  x$ the two ground states. Then it holds:
$$
1-\frac{2b}{2b^2+b-1}<\tilde{x}^2<1-e^{-2b}\,.
$$
Moreover, for any  $x \in (0,1)$, it holds
\begin{equation}
\gCW_b(x)\ge \gCW_b(\tilde x) +
\frac{b-1}{2b}(|x-\tilde x|\land \eps)^2,\label{EQ:global-quad}
\end{equation}
and
\begin{equation}
\gCW_b(x)\ge \gCW_b(-\tilde x) +
\frac{b-1}{2b}(|x+\tilde x|\land \eps)^2,\label{EQ:global-quad2}
\end{equation}
where $\eps =\frac{e^{-2b}}{4}\big(1-\frac1b\big)$.

Fix $b\le 1$ in the Curie-Weiss model and recall that $\tilde x=0$ is the unique ground state. Then for any $x \in (-1,1)$  it holds
\begin{equation}
\label{EQ:global-quad3}
\gCW_b(x)\ge \gCW_b(0) + (1-b)(x\wedge \eps')^2\,.
\end{equation}
where 
$$
\eps'= \sqrt{\frac{1-b}{3}}\,.
$$
\end{lemma}
\begin{proof}
Observe that for $x > 0$, we have 
\begin{equation}
\label{EQ:approxlog}
2b \tilde x=\log\big(\frac{1+\tilde x}{1-\tilde x}\big) <\frac{2\tilde x}{1-\tilde{x}^2}-\gamma \tilde{x}^3\,, \quad \forall \gamma \le 1\,.
\end{equation}
Taking $\gamma=0$ implies that $\tilde x>\sqrt{1-1/b}$. Plugging this into~\eqref{EQ:approxlog} with $\gamma=1$ yields
$$
2b \tilde x < \frac{2\tilde x}{1-\tilde{x}^2}- \tilde{x} \big(1-\frac{1}{b}\big)\,.
$$
Solving for $\tilde x$ once again yields
\begin{equation}
\label{EQ:lbspontmag}
\frac{2}{1-\tilde{x}^2} >2b + 1-\frac{1}{b}
\end{equation}
Or equivalently that 
$$
\tilde{x}^2>1-\frac{2b}{2b^2+b-1}\,.
$$
Moreover, the mean field equation yields 
$$
2b > 2b \tilde x=\log\big(\frac{1+\tilde x}{1-\tilde x}\big) >-\log(1-\tilde x)
$$
so that  
\begin{equation}
\label{EQ:ubspontmag}
\tilde x < 1-e^{-2b}
\end{equation}
which readily yields the desired upper bound on $\tilde x^2$.

We conclude this proof by showing that $\gCW_b$ is at least quadratic in a neighborhood of its minima when $b \neq 1$. To that end, observe first that the second and third derivatives of $g$ are given respectively by
$$
\frac{\partial^2}{\partial x^2}\gCW_b(x)=-4b +\frac{4}{1-x^2}\,, \qquad \frac{\partial^3}{\partial x^3}\gCW_b(x)=-\frac{8x}{(1-x^2)^2}\,,
$$

First assume that $b>1$.  A Taylor expansion of $\gCW_b$
around $\tilde{x}$ together with~\eqref{EQ:lbspontmag}
and~\eqref{EQ:ubspontmag} yields that for any $\eps \in (0,1)$ and
$x$ such that 
$$
|x-\tilde x|\leq \eps:=\frac{e^{-2b}}{2}\wedge \big(1-\frac1b\big)\,,
$$
\begin{align*}
\gCW_b(x)&\ge \gCW_b(\tilde x) + \big(1-\frac{1}{b}\big)(x-\tilde x)^2 -\frac{4}{3(1-(\tilde x + \eps)^2)^2}|x-\tilde x|^3\\
&\ge \gCW_b(\tilde x) + \big(1-\frac{1}{b}\big)(x-\tilde x)^2  -\frac{4}{3(1-\tilde x-  \eps)}|x-\tilde x|^3\\
&\ge \gCW_b(\tilde x) + \big(1-\frac{1}{b}\big)(x-\tilde x)^2  -\frac{4\eps}{3(e^{-2b}-  \eps)}(x-\tilde x)^2\\
&\ge \gCW_b(\tilde x) + \frac12\big(1-\frac{1}{b}\big)(x-\tilde x)^2 \,.
\end{align*}

Now, using the fact that $\gCW_b$ is monotone decreasing on $(0,
\tilde{x} - \eps)$ and monotone increasing in $(\tilde{x} +
\eps, 1)$, we obtain the claim in \eqref{EQ:global-quad}. The lower bound~\eqref{EQ:global-quad2} follows by symmetry.

Next, assume that $b<1$.  A Taylor expansion of $\gCW_b$ around  $0$ yields that for any $x$ such that $|x|<\eps', \eps' \in (0,1)$,
\begin{align*}
\gCW_b(x)&> \gCW_b(0) + \big[2(1-b)-\frac{4\eps^2}{3(1-\eps^2)^2}\big]x^2 \\
&\ge \gCW_b(0) + (1-b)x^2
\end{align*}
for  
$$
\eps' \le \sqrt{\frac{1-b}{3}}\,.
$$
Using the fact that $\gCW_b$ is monotone decreasing on $[1, -\eps)$ and monotone increasing on $(\eps, 1]$  yields~\eqref{EQ:global-quad3}.
\end{proof}
\begin{remark}
When $b=1$, the Hessian of $\gCW_b$ vanishes at $0$. In this case, $\gCW_b$ is not lower bounded by a quadratic term.
\end{remark}

\section{Inequalities}
\label{sec:some-inequalities}

\subsection{Bounds on binomial coefficients}
\label{sec:bounds-binom-coeff}
We need the following well known information theoretic
estimate. Recall that the binary entropy function $h:[0,1] \to \R$ is defined by $h(0)=h(1)=0$ and for any $s \in (0,1)$ by
$$
h(s)=-s\log(s) - (1-s)\log(1-s)\,.
$$

\begin{lemma}%
\label{lem:binom-bound}
Let $m$ be a
  positive integer and let $\gamma \in [0, 1]$ be such that $\gamma m$
  is an integer.  Then $$\binom{m}{\gamma m} \leq
  \exp(mh(\gamma))\,.$$%
\end{lemma}
\begin{proof}
Let $X\sim \mathsf{Bin}(n,\gamma)$ be a binomial random variable. Then
$$
1\ge \p(X=\gamma m)={m \choose \gamma m}\gamma^{\gamma m}(1-\gamma)^{(1-\gamma)m}={m \choose \gamma m}\exp(-mh(\gamma))\,.
$$

\end{proof}

The following sharper estimate follows from the Stirling approximation of $n!$ developed in~\cite{Rob55}.
\begin{lemma}%
\label{lem:binom-stirling}
  Let $\eps > 0$, $m$ a positive integer let
  $\gamma \in [\eps, 1-\eps]$ be such that $\gamma m$ is an
  integer.  We then have
  \begin{displaymath}
    \exp\inp{-\frac{1}{12\eps^2m}} \leq \sqrt{2\pi m \gamma (1-\gamma)}\exp(mh(\gamma))\binom{m}{\gamma
      m} \leq \exp\inp{\frac{1}{12m}}\,.
  \end{displaymath}
\end{lemma}
\begin{proof}%
 It follows from~\cite{Rob55}) that for any positive integer $n$,
  \begin{displaymath}
    1 \leq \exp\inp{\frac{1}{12n + 1}}  \leq \frac{n!}{\sqrt{2\pi n}(n/e)^n}
    \leq \exp\inp{\frac{1}{12n}}.
 \end{displaymath}
 Applying this to
 $$\binom{m}{\gamma m} = \frac{m!}{(\gamma m)! ((1-\gamma)m)!}$$
 yields the desired bounds.
\end{proof}

\subsection{Tail bound for the $\chi^2$ distribution}
We recall here a well known tail bound for the $\chi^2$ distribution \citep[see][Lemma~1]{LauMas00}.
\begin{lemma}
\label{lem:laumas00}
Let $Z\sim \cN_2(0, I_2)$ be a bivariate standard Gaussian vector. Then, for any $t \ge 2$, it holds
$$
\p(\|Z\|_2^2 -2 \ge 2)\le \exp(-t/4)\,.
$$
\end{lemma}
\label{sec:gaussian-integral}

\end{document}